\documentclass[12pt]{article}

\pdfoutput=1
\usepackage{amsmath,amsthm,amssymb,amsfonts,bbm,framed,mathrsfs}
\usepackage{subcaption,float}
\usepackage{algorithm,algorithmic}
\usepackage{verbatim}
\usepackage{graphicx}
\usepackage{bbm}
\usepackage[english]{babel}
\usepackage{enumitem}
\usepackage{apacite}
\usepackage{indentfirst}

\newcounter{dfn}
\newcounter{rem}

\newtheorem{theorem}{Theorem}

\newtheorem{lemma}[theorem]{Lemma}
\newtheorem{proposition}[theorem]{Proposition}
\newtheorem{definition}[dfn]{Definition}

\theoremstyle{remark}
\newtheorem{remark}[rem]{Remark}

\renewcommand{\epsilon}{\varepsilon}

\newcommand{\defined}{\triangleq}

\DeclareMathOperator{\rank}{rank}
\DeclareMathOperator{\Tr}{tr}
\DeclareMathOperator{\diag}{diag}

\newcommand{\conp}{\overset{\p}{\to}}
\newcommand{\cond}{\overset{\mathcal{L}}{\to}}

\newcommand{\equid}{\overset{\mathcal{L}}{=}}
\newcommand{\op}{o_\p}
\newcommand{\Op}{\mathcal{O}_\p}

\newcommand{\var}{\text{Var}}
\newcommand{\e}{\mathbb{E}}

\newcommand{\n}{\mathcal{N}}
\newcommand{\p}{\mathbb{P}}

\newcommand{\A}{\mathcal{A}}

\newcommand{\C}{\mathsf{c}}

\newcommand{\m}{\mathcal{M}}

\newcommand{\R}{\mathbb{R}}
\newcommand{\N}{\mathbb{N}}

\newcommand{\T}{\mathsf{T}}

\newcommand{\Tset}{\mathscr{T}}

\newcommand{\pp}[1]{{P}^{\perp}_{#1}}

\newcommand{\phidelta}{\phi_\delta}

\newcommand{\betahatew}{\hat{\beta}_{\text{EW}}}

\newcommand{\few}{F_{\text{EW}}}

\newcommand{\fld}{F_\text{LD}}

\newcommand{\betahatewQ}{\hat{\beta}_{\text{EW},Q}}

\newcommand{\betahatmQ}{\hat{\beta}_{m,\text{Q}}}

\newcommand{\sigmanuhat}{\hat{\sigma}_\nu^2}
\newcommand{\sigmanuhatQ}{\hat{\sigma}_{\nu,Q}^2}
\newcommand{\sigmanuhatld}{\hat{\sigma}_{\nu,\text{LD}}^2}
\newcommand{\sigmaepsilonhat}{\hat{\sigma}_\epsilon^2}
\newcommand{\sigmaepsilonhatQ}{\hat{\sigma}_{\epsilon,Q}^2}
\newcommand{\sigmagammahat}{\hat{\sigma}_\gamma^2}

\newcommand{\sigmanuhatls}{\hat{\sigma}_{\nu,\text{LS}}^2}
\newcommand{\sigmaepsilonhatls}{\hat{\sigma}_{\epsilon,\text{LS}}^2}
\newcommand{\sigmagammahatls}{\hat{\sigma}_{\gamma,\text{LS}}^2}

\newcommand{\fls}{F_\text{LS}}

\newcommand{\sigmaahat}{\hat{\sigma}_a^2}
\newcommand{\sigmabhat}{\hat{\sigma}_b^2}

\newcommand{\etahat}{\hat{\eta}_\text{EW}}
\newcommand{\etahatls}{\hat{\eta}_\text{LS}}
\newcommand{\etahatoracle}{\hat{\eta}_\text{oracle}}

\newcommand{\rem}{\Delta}

\newcommand{\betahatm}{\hat{\beta}_m}

\newcommand{\dbar}{\bar{d}}
\newcommand{\lambdabar}{\bar{\lambda}}

\begin{document}
	\title{Estimating the Random Effect in Big Data Mixed Models}
	
	\author{Michael Law \and Ya\hspace{-.1em}'\hspace{-.1em}acov Ritov \thanks{Supported in part by NSF Grant DMS-1712962.}}
	\date{
		University of Michigan\\
		\today
	}
	\renewcommand\footnotemark{}
		
	\maketitle
	
	\begin{abstract}
		
		We consider three problems in high-dimensional Gaussian linear mixed models.  Without any assumptions on the design for the fixed effects, we construct an asymptotic $F$-statistic for testing whether a collection of random effects is zero, derive an asymptotic confidence interval for a single random effect at the parametric rate $\sqrt{n}$, and propose an empirical Bayes estimator for a part of the mean vector in ANOVA type models that performs asymptotically as well as the oracle Bayes estimator.  We support our results with numerical simulations and provide comparisons with oracle estimators.  The procedures developed are applied to the Trends in International Mathematics and Sciences Study (TIMSS) data.
		
	\end{abstract}

	\section{Introduction}
	In the past two decades, there has been a lot of progress in the theory for high-dimensional linear models.  However, its close cousin, the high-dimensional linear mixed model, has received significantly less attention; it was not until the past decade until there were procedures for estimation.  Consider a linear mixed model given by
	\begin{align}\label{modellmmy}
		Y = \mu + Z\nu + W\gamma + \epsilon,
	\end{align}
	with $Z \in \R^{n\times v}$, $W\in \R^{n\times r}$, and $Y,\mu,\epsilon \in \R^n$.  In addition, we observe covariates $X\in\R^{n\times p}$ such that $\rem \defined \mu - X\beta$ is a small remainder term for some sparse vector $\beta\in\R^p$ (see Section \ref{sectionnotation} for a rigorous definition).  Therefore, the model may be written as
	\begin{align*}
		Y = X\beta + Z\nu + W\gamma + \rem + \epsilon.
	\end{align*}
	Here, $X$ is a design corresponding to the fixed effects and $(Z,W)$ to the random effects.  We consider the setting where the random effects are low-dimensional, $v+r < n$, but the fixed effects are high-dimensional, $p>n$.  We have separated the random effects in two to emphasize later that we are interested in $\nu$ and view $\gamma$ as nuisance parameters.  Various authors have considered different aspects of this problem.  
	
	The earliest work of \citeA{schelldorfer2011} proposed an estimator for both $\beta$ and the variance components using a lasso-type approach.  These types of approaches were later extended by several authors who considered estimation with both convex, such as \citeA{groll2014}, and non-convex penalties, such as \citeA{wang2012}.  There has also a growing literature on model selection in high-dimensional linear mixed models (cf. \citeA{muller2013} and the references therein).  
	
	The problem of inference is slightly less well studied.  To the best of our knowledge, \citeA{chen2015} was the first to consider hypotheses testing problems for random effects in ANOVA type settings and \citeA{bradic2017} for fixed effects in high-dimensional models.  During the preparation of this manuscript, we became aware of the very recent work of \citeA{li2019}, who consider the problem of inference in high-dimensional linear mixed models.  In particular, they discuss inference for fixed effects and estimation of variance components.
	
	The goal of the present paper is to contribute to this growing literature on high-dimensional linear mixed models, both in terms of estimation and inference when the random effects and error are all Gaussian.  In particular, we consider three related problems:
	
	\begin{enumerate}
		\item Testing whether $\nu \equiv 0$.
		\item Constructing confidence intervals for $\sigma_\nu^2$ when $\nu \sim \n_v\left( 0_v , \sigma_\nu^2 I_v \right)$.
		\item Estimating using empirical Bayes in ANOVA Type Models.
	\end{enumerate}
		
	Our methodology is inspired by both low-dimensional linear mixed models as well as high-dimensional linear models.  Specifically, our approach to all three problems starts with considering a procedure in the corresponding low-dimensional problem and retrofitting it with tools and techniques from high-dimensional linear models to produce a procedure for the high-dimensional linear mixed model.
	
	\subsection{Organization of the Paper}
	We will end the current section with a description of the notation that we will adopt throughout the paper.  Sections \ref{sectiontesting}, \ref{sectionci}, and \ref{sectionestimation} consider each of the three problems outlined in the Introduction respectively.  Each one starts with a description of the problem setup, a brief motivation from the low-dimensional problem, a description of the estimator that will be considered, and ends with some theoretical results.  In Section \ref{sectionnumerics}, we provide an overview of the computation of the estimators as well as simulations and a real data application.  We conclude with discussions in Section \ref{sectiondiscussion}.  For the ease of presentation, we defer all proofs to Section \ref{sectionproofs}.
	
	\subsection{Notation}\label{sectionnotation}
	Throughout, all of our variables will have a dependence on $n$.  However, when this will not cause confusion, we will suppress this dependence.  Let $\left\Vert \cdot \right\Vert$ denote the standard (unscaled) Euclidean norm with the dimension of the space being implicit from the vector and $\left\Vert \cdot \right\Vert_0$ the $L_0$-``norm'' on $\R^p$.  We will also use $\left\Vert \cdot \right\Vert$ to denote the $\ell_2$ norm for matrices and $\left\Vert \cdot \right\Vert_\text{HS}$ to denote the Hilbert-Schmidt norm.  Furthermore, $\lambda_\text{max}(\cdot)$ and $\lambda_\text{min}(\cdot)$ will denote the maximal and minimal eigenvalue respectively.  
	
	Consistent with other high-dimensional works, we will assume that $\beta$ is a sparse vector.  There are various notions of sparsity but we will assume the slightly more general setting of weak sparsity.  For $u \in \N$, we will let $\m_u$ denote the collection of all models with the dimension of the fixed effects design equal to $u$.  For a model $m\in\m_u$, $X_m$ will denote the $n\times u$ sub-matrix of $X$ corresponding to the columns indexed by $m$.  Moreover, $P_m$ will denote the projection onto $X_m$ and $\pp{m}$ the projection onto the orthogonal complement.  We will consider the following definition of weak sparsity.
	
	\begin{definition}
		The vector $\mu$ is said to satisfy the \emph{weak sparsity property relative to $X$} with sparsity $s$ at rate $k$ if the set 
		\begin{align*}
		\mathcal{S}_\mu \defined \left\{m\in\m_s : \left\Vert \pp{m}\mu \right\Vert^2 = o(k) \right\}
		\end{align*}
		is non-empty.  A set $S\in \mathcal{S}_\mu$ is said to be a \emph{weakly sparse set} for the vector $\mu$.
	\end{definition}		
	
	Then, we will let $S\in\mathcal{S}_\mu$ denote any weakly sparse set for $\mu$.
	
	\section{A High-Dimensional $F$-Test: Testing $\nu\equiv 0$}\label{sectiontesting}

	\subsection{Model and Motivation}
	Consider the high-dimensional linear mixed model given in equation (\ref{modellmmy}).  For the fixed effects $X$, we will make no assumption on the design besides independence with $(\nu,\gamma,\epsilon)$.  For the random effects, we will similarly make no assumption on $(W,Z)$, except for the fact that $(W,Z)$ is independent of $(\nu,\gamma,\epsilon)$.  We will assume that $\epsilon \sim \n_n(0_n,\sigma_\epsilon^2 I_n)$ for some positive constant $\sigma_\epsilon^2>0$ and $\nu \sim \n_v(0_v,\Psi)$ for some symmetric positive semi-definite matrix $\Psi$.  No assumption will be necessary on $\gamma$ as the nuisance parameters will be projected out in the first stage.
	
	We are interested in the hypotheses testing problem
	\begin{align}\label{testinghypotheses}
		H_0:  \lambda_\text{max}(\Psi) = 0 &&& H_1:  \lambda_\text{max} (\Psi) > 0.
	\end{align}
	
	Suppose temporarily that we are in the low-dimensional setting with $s = p$, $p + v + r < n$, and $\rem \equiv 0$.  Then, in this low-dimensional problem, with $\nu$ and $\epsilon$ both distributed as Gaussian, the standard procedure for testing whether $\nu\equiv 0$ is through the $F$-test.  Letting $\tilde{A}$ denote the projection onto $Z \ominus (X,W)$ with $n_{\tilde{a}} = \rank(\tilde{A})$ and $\tilde{B}$ denote the projection onto $(X,Z,W)^\perp$ with $n_{\tilde{b}}=\rank(\tilde{B})$, we have that 
	\begin{align}\label{formulaflowdimensional}
		F = \frac{\left\Vert \tilde{A} Y \right\Vert^2/{n_{\tilde{a}}}}{\left\Vert \tilde{B} Y \right\Vert^2/{n_{\tilde{b}}}}.
	\end{align}
	
	Then, under the null hypothesis, the above statistic has an $F_{n_{\tilde{a}},n_{\tilde{b}}}$ distribution.  It is known that the $F$-test has certain optimality properties for certain classes of hypotheses testing problems (cf \citeA{jiang2007} and the references therein for more details).
	
	The main obstacle to directly using this $F$-test in the high-dimensional setting is removing the contribution of the fixed effects.  One possibility is to perform model selection and choose the relevant covariates from $X$ and then use the low-dimensional $F$-test.  \citeA{chen2015} consider a similar problem in the high-dimensional setting and they use a SCAD based approach for variable selection.  As a consequence, they require $p=o(\sqrt{n})$.  Instead, we leverage the fact that a projection onto a particular space is a regression onto a design whose columns span that same space.
	
	Expanding both the numerator and the denominator of the classical $F$-statistic, we have that
	\begin{align*}
		&\tilde{A}Y = \tilde{A}Z\nu + \tilde{A}\epsilon\\
		&\tilde{B}Y = \tilde{B}\epsilon.
	\end{align*}
	
	In both matrices above, they project onto the orthogonal complement of $W$, which may still be achieved in the high-dimensional problem since $W$ is a low-dimensional matrix.  We may find two matrices, $A$ and $B$, such that
	\begin{align}
		&AY = AX\beta + AZ\nu + A\epsilon\label{modelynum},\\
		&BY = BX\beta + B\epsilon\label{modelyden},
	\end{align}
	where $A \in \R^{{n_a}\times n}$ and $B \in \R^{{n_b}  \times n}$ have orthonormal rows that are mutually orthogonal (ie. $AB^\T = 0_{n_a\times n_b}$).  For simplicity, we will write 
	\begin{align*}
		Q \defined \begin{pmatrix} A \\ B \end{pmatrix}.
	\end{align*}
	Rewriting, we have that
	\begin{align}\label{modelyq}
		\begin{pmatrix} AY \\ BY \end{pmatrix} = QY = QX\beta + QZ\nu + Q\epsilon = \begin{pmatrix} AX\beta + AZ\nu + A\epsilon \\ BX\beta + B\epsilon \end{pmatrix}.
	\end{align}
	If $QX$ was low-dimensional, to obtain the projection of $QY$ onto the orthogonal complement of $QX$, this is equivalent to finding the residuals of $QY$ using the covariates $QX$; this yields $QY - QX\hat{\beta}$.  Then, we have that
	\begin{align*}
		\begin{pmatrix} AY - AX\hat{\beta} \\ BY - BX\hat{\beta} \end{pmatrix}  &= QY - QX\hat{\beta} = QX\beta - QX\hat{\beta} + QZ\nu + Q\epsilon\\ &= \begin{pmatrix} AX\beta - AX\hat{\beta} \\ BX\beta - BX\hat{\beta} \end{pmatrix} + \begin{pmatrix} AZ\nu + A\epsilon \\ B\epsilon \end{pmatrix}.
	\end{align*}
	
	Hence, this recasts the problem into one of high-dimensional prediction, for which there have been many procedures suggested to estimate $QX\beta$, such as the lasso and exponential weighting (cf. \citeA{tibshirani1996} and \citeA{leung2006} respectively).  Therefore, we propose using a plug-in estimator for $QX\beta$ using exponential weighting of all models of a particular size and then consider the resultant residuals.  This idea, under mild assumptions, will provide an asymptotic $F$-test.
	
	\subsection{Estimator}
	
	Rather than using exponential weighting for $AY$ and $BY$ separately, we will estimate the mean vector jointly; the advantages to this approach are twofold.  On the one hand, it will significantly simplify our implementation as we will only need a single Markov chain (see Section \ref{sectionnumerics}) and, on the other, the estimation seems to perform better empirically since we pool more observations together.  Therefore, we will fix a sequence of constants $u\in\N$.  Let $\betahatmQ$ denote the least-squares estimator of $\beta$ using the model $m\in\m_u$ with covariates $QX_m$.  We will define our set of exponential weights by
	\begin{align*}
		w_{m,Q} \defined \frac{\exp\left( -\frac{1}{\alpha} \left\Vert QY - QX\betahatmQ \right\Vert^2 \right)}{\sum_{k\in\m_u} \exp\left( -\frac{1}{\alpha} \left\Vert QY - QX\hat{\beta}_{k,Q} \right\Vert^2 \right)},
	\end{align*}
	where $\alpha \geq 6\lambda_\text{max}\left( QZZ^\T Q^\T + \sigma_\epsilon^2 I_{n_a+n_b}\right)$.  We will estimate $\beta$ by 
	\begin{align*}
		\betahatewQ \defined \sum_{m\in\m_u} w_{m,Q} \betahatmQ.
	\end{align*}
	The bound on $\alpha$ is to ensure the consistency result of Lemma \ref{lemmaoraclecorrelated} under both the null and the alternative hypotheses.  The subscript $Q$ is to emphasize that the estimator is constructed only using the $Q$ projected sub-data from (\ref{modelyq}).  Then, the corresponding $F$-statistic will be
	\begin{align*}
		\few \defined \frac{\left\Vert AY - AX\betahatewQ \right\Vert^2/{n_a}}{\left\Vert BY - BX\betahatewQ \right\Vert^2/{n_b}}.
	\end{align*}
	Similar to the classical $F$-statistic, we will reject the null hypothesis for large values of $\few$.  In particular, for a value $\delta\in(0,1)$, let $F_{n_a,n_b,\delta}$ denote the $\delta$ upper quantile of an $F_{n_a,n_b}$ distribution.  Then, we will consider tests of the form
	\begin{align*}
		\phidelta \defined \mathbbm{1}\left(\few > F_{n_a,n_b,\delta}\right).
	\end{align*}
	
	\subsection{Assumptions}
	We will assume without the loss of generality that the columns of $X$ have squared norm that is $\mathcal{O}(n)$.  Moreover, we will assume that
	\begin{enumerate}[label=(A\arabic*)]
		\item \label{assumptiondistnu} The vector $\nu$ satisfies
		\begin{align*}
			\nu \sim \n_v(0_v,\Psi).
		\end{align*}
		Furthermore, $\lambda_\text{max}\left( QZ \Psi Z^\T Q^\T + \sigma_\epsilon^2 I_n \right) = \mathcal{O}(1)$.
		\item \label{assumptionsparsitynull} The sequence of constants $u$ satisfies
		\begin{align*}
			&\liminf_{n\to\infty} (u-s) \geq 0,\\
			&\lim_{n\to\infty} \frac{u\log(p)}{n_a \wedge n_b} \to 0.
		\end{align*}
	\end{enumerate}
	\begin{enumerate}[label=(A$\arabic*^\ast$)]
		\setcounter{enumi}{1}
		\item \label{assumptionsparsityalternative} The sequence of constants $u$ and the number of observations in the reduced models, $n_a$ and $n_b$, satisfy
		\begin{align*}
			&\liminf_{n\to\infty} (u-s) \geq 0,\\
			&\lim_{n\to\infty} \frac{u\log(p)}{\sqrt{n}} = 0,\\
			&n_a \asymp n_b \asymp n.
		\end{align*}
	\end{enumerate}
	
	\begin{remark}
		There are two assumptions regarding the sparsity and the number of random effects.  The first assumption, \ref{assumptionsparsitynull}, is quite weak in terms of the sparsity level and is used to establish the distribution of the estimator under the null hypothesis.  Considering the models in equations (\ref{modelynum}) and (\ref{modelyden}), which have $n_a$ and $n_b$ observations respectively, the sparsity only needs to be of smaller order than the number of observations in each of the two reduced models.  This is known to be the optimal rate at which the mean vector can be estimated in a high-dimensional linear model (for example, cf. \citeA{raskutti2011} or \citeA{rigollet2011}).  The requirement that $\liminf (u-s)\geq 0$ is only to ensure that the models we consider for exponential weighting are sufficiently large to remove the fixed effects. 
		
		The stronger assumption, \ref{assumptionsparsityalternative}, is used when considering the power of the test under a specific contiguous alternative (see equation (\ref{testinghypothesescontiguous})), which gives the parametric rate of $\sqrt{n}$.  However, in cases where we have a different sequence of alternatives, the assumption may be weakened accordingly.  Note that this is similar to the setting of the de-biased lasso, which requires a ultra sparsity to achieve the parametric rate in testing a single covariate (cf. \citeA{cai2017}).  
	\end{remark}
	
	\subsection{Main Results for $\few$}
	We start by stating the asymptotic distribution of $\few$ under the null hypothesis.
	\begin{theorem}\label{theoremnullhypothesis}
		Consider the linear mixed model given by equation (\ref{modellmmy}) and the hypotheses testing problem from equation (\ref{testinghypotheses}).  Assume \ref{assumptionsparsitynull}.  Under the null hypothesis and $\alpha \geq 4\sigma_\epsilon^2$,
		\begin{align*}
			\few = F + \op(1),
		\end{align*}
		where $F\sim F_{n_a,n_b}$.
	\end{theorem}

	This naturally leads to the question regarding the power of the testing procedure.  In particular, between what contiguous alternatives can we distinguish?  We will consider the following contiguous testing problem.
	\begin{align}\label{testinghypothesescontiguous}
		H_0:  \lambda_{\text{max}}(AZ\Psi Z^\T A^\T) = 0 &&& H_1:  \lambda_{\text{min}}(AZ\Psi Z^\T A^\T) = \frac{h}{\sqrt{n}}.
	\end{align}
	
	\begin{remark}
		Suppose that $\nu$ corresponds to a single random effect and the design is balanced, with $\nu \sim \n_v(0_v,\sigma_\nu^2 I_v)$.  Then, the above hypotheses becomes
		\begin{align*}
			H_0:  \sigma_\nu^2 = 0 &&& H_1:  \sigma_\nu^2 = \frac{h}{\sqrt{n}},
		\end{align*}
		which is a standard hypotheses testing problem, such as in the balanced one-way random effects model.  In this model, in the low-dimensional setting, the rate of $\sqrt{n}$ is optimal.
	\end{remark}
	
	\begin{theorem}\label{theoremalternativehypothesis}
		Consider the linear mixed model given by equation (\ref{modellmmy}) and the hypotheses testing problem from equation (\ref{testinghypothesescontiguous}).  Assume further \ref{assumptiondistnu} and \ref{assumptionsparsityalternative}.  Fix a value of $\delta>0$.  Under the alternative hypothesis with $h>0$ sufficiently large (not depending on $n$) and $\alpha \geq 6\lambda_\text{max}\left( QZ\Psi Z^\T Q^\T + \sigma_\epsilon^2 I_{n_a+n_b}\right)$, the sum of type I and type II error for the test statistic $\phidelta$ is between zero and one.
	\end{theorem}
	
	To prove the theorem, we will need the following lemma that provides consistency of $\betahatewQ$ under the setting of correlated errors, which may be of interest in its own right.
	
	\begin{lemma}\label{lemmaoraclecorrelated}
		Consider a high-dimensional linear model given by
		\begin{align*}
			Y = \mu + \xi,
		\end{align*}
		with $\xi \sim \n_n(0_n,\Sigma)$ for some covariance matrix $\Sigma$ satisfying $\lambda = \lambda_\text{max}(\Sigma) = \mathcal{O}(1)$.  Assume that $\mu$ is weakly sparse relative to $X$ with sparsity $s$ at rate $n$ and $\left\Vert X\beta \right\Vert^2 = \mathcal{O}(n)$.  Assume further $\log(|\m_u|) = o(n)$.  Letting $\betahatm$ denote the least-squares estimator for $\beta$ using the covariates $X_m$, define the exponential weights as
		\begin{align*}
			w_m \defined \frac{\exp\left( -\frac{1}{\alpha} \left\Vert Y - X\betahatm \right\Vert^2 \right)}{\sum_{k\in\m_u} \exp\left( -\frac{1}{\alpha} \left\Vert Y - X\hat{\beta}_k \right\Vert^2 \right)},
		\end{align*}
		with $\alpha > 6\lambda$.  Then,
		\begin{align*}
			\frac{1}{n} \e \left\Vert \sum_{m\in\m_u} w_m X\betahatm - \mu \right\Vert^2 \to 0.
		\end{align*}
		If in addition $\mu$ is weakly sparse relative to $X$ with sparsity $s$ at rate $\sqrt{n}$ and $\log(|\m_u|) = o(\sqrt{n})$, then
		\begin{align*}
			\frac{1}{\sqrt{n}} \e \left\Vert \sum_{m\in\m_u} w_m X\betahatm - \mu \right\Vert^2 \to 0.
		\end{align*}
	\end{lemma}
	
	\begin{remark}
		The above lemma slightly generalizes Theorem 5 of \citeA{leung2006} to heteroskedastic and dependent, but still Gaussian, errors.  However, in the setting where the errors are independent and identically distributed, the result of \citeA{leung2006} is stronger as their bound for $\alpha$ is smaller.
	\end{remark}

	\section{Confidence Intervals for $\sigma_\nu^2$}\label{sectionci}
	
	\subsection{Model and Motivation}

	In the previous section, we considered the problem of testing a collections of random effects.  However, it is often of interest to construct confidence intervals for the variance of a particular random effect.  Suppose that $\Psi = \sigma_\nu^2 I_v$.  In the low-dimensional setting, there have been many procedures suggested to construct confidence intervals, from likelihood based approaches to $F$-test inversions (for example, see \citeA{jiang2007} and the references therein for a non-exhaustive list).
	
	We take an approach similar to variance estimation in linear models.  Consider the two reduced models from Section \ref{sectiontesting}, which are reproduced below for convenience:
	\begin{align*}
		&AY = AX\beta + AZ\nu + A\epsilon,\\
		&BY = BX\beta + B\epsilon,
	\end{align*}
	
	The second model is a standard high-dimensional linear model with $B\epsilon \sim \n_{n_b}\left(0_{n_b},\sigma_\epsilon^2 I_{n_b} \right)$.  In the low-dimensional setting, $\sigma_\epsilon^2$ would be estimated by mean squared error, which is computed through projecting $BY$ onto the orthogonal complement of $BX$.  Hence, we may achieve a similar result by viewing this as a prediction problem and regressing out $BX$.  For the other model, note that we have a high-dimensional linear model with noise $AZ\nu + A\epsilon \sim \n_{n_a} \left( 0_{n_a} , \sigma_\nu^2 AZZ^\T A^\T + \sigma_\epsilon^2 I_{n_a} \right)$.  Since the matrix $A$ and $Z$ are both observed, we may whiten the data with respect to $\nu$ by taking the Spectral Decomposition of $AZZ^\T A^\T = V D V^\T$ and left multiplying by $D^{-1/2} V^\T$.  Here, $D$ is diagonal and $V$ is unitary.  This yields
	\begin{align}\label{modelyci}
		&D^{-1/2}V^\T A Y = D^{-1/2}V^\T AX\beta + D^{-1/2}V^\T AZ\nu + D^{-1/2}V^\T A \epsilon .
	\end{align}
	Then, the above can be viewed as a high-dimensional linear model with noise given by $D^{-1/2}V^\T AZ\nu + D^{-1/2}V^\T A \epsilon \sim \n_{n_a}\left( 0_{n_a} , \sigma_\nu^2 I_{n_a} + \sigma_\epsilon^2D^{-1} \right)$.  Since $D^{-1}$ is observed and $\sigma_\epsilon^2$ can be estimated, this will provide an estimator of $\sigma_\nu^2$, which we may then use to construct a pivot.
	
	\subsection{Estimator}
	As outlined previously, we will need to construct two estimators, one for $\sigma_\epsilon^2$ and one for $\sigma_\nu^2$.  Again, we will jointly estimate the mean vector $Q\mu$.  Therefore, we will start by defining
	\begin{align*}
		\sigmaepsilonhatQ \defined \frac{1}{n_b} \left\Vert BY - BX\betahatewQ \right\Vert^2.
	\end{align*}
	Then, $\sigmanuhatQ$ will be given by
	\begin{align*}
		\sigmanuhatQ \defined \frac{1}{\Tr(D^{-1})}\left\Vert D^{-1/2}V^\T AY - D^{-1/2} V^\T AX\betahatewQ \right\Vert^2 - \sigmaepsilonhatQ.
	\end{align*}
	
	\subsection{Assumptions}
	In addition to the assumptions from Section \ref{sectiontesting}, we will need two additional assumptions regarding the distribution of the values of the diagonal matrix $D$.
	
	\begin{enumerate}[label=(B\arabic*)]
		\item \label{assumptioneffectivesamplesize} Letting $d=\diag(D)$, the norms satisfy
		\begin{align*}
			\limsup_{n\to\infty} \frac{\left\Vert \sigma_\nu^2 1_{n_a} + \sigma_\epsilon^2 d \right\Vert_\infty}
			{\left\Vert \sigma_\nu^2 1_{n_a} + \sigma_\epsilon^2 d \right\Vert_2} \to 0.
		\end{align*}

		\item \label{assumptionasymptoticsamplesize} The sample sizes $n_a$ and $n_b$ satisfy
		\begin{align*}
			&\lim_{n\to\infty} \frac{n\left\Vert \sigma_\nu^2 1_{n_a} + \sigma_\epsilon^2 d \right\Vert_2^2}{\left(\Tr(D^{-1})\right)^2} \to \sigma_a^2>0,\\
			&\lim_{n\to\infty} \frac{2\sigma_\epsilon^4n}{n_b} \to \sigma_b^2>0,\\
			&\lim_{n\to\infty} \frac{n_a}{\Tr(D^{-1})}\to \dbar > 0.
		\end{align*}
	\end{enumerate}
	
	\begin{remark}
		Assumption \ref{assumptioneffectivesamplesize} on the matrix $D$ is to ensure that the design is sufficiently well balanced and that $Z$ is distinct enough from $W$.  This assumption is satisfied in general settings such as relatively balanced ANOVA models.

		Assumption \ref{assumptionasymptoticsamplesize} is a technical regularity condition in view of the triangular array framework.  By possibly reducing to a convergent sub-sequence, we may always satisfy assumption \ref{assumptionasymptoticsamplesize}.
	\end{remark}
	
	\subsection{Main Results for $\sigmanuhat$}
	We will start by stating the asymptotic distribution of $\sigmanuhat$.
	
	\begin{theorem}\label{theoremci}
		Consider the linear mixed model given by equation (\ref{modellmmy}), with the reduced models given in equations (\ref{modelyden}) and (\ref{modelyci}).  Assume further \ref{assumptiondistnu} with $\Psi = \sigma_\nu^2 I_v$, \ref{assumptionsparsityalternative}, \ref{assumptioneffectivesamplesize}, and \ref{assumptionasymptoticsamplesize}.  Then,
		\begin{align*}
			\sqrt{n} \left( \sigmanuhatQ - \dbar\sigma_\nu^2 \right) \cond \n(0,\sigma_a^2 + \sigma_b^2).
		\end{align*}
	\end{theorem}
	
	From the previous theorem, we immediately have that the following is a $(1-\alpha)$ confidence interval for $\sigma_\nu^2$.
	\begin{align}\label{equationci}
		\left( \frac{1}{\dbar} \left( \sigmanuhatQ - \frac{\sqrt{\sigma_a^2+\sigma_b^2}}{\sqrt{n}} z_{\alpha/2} \right) , \frac{1}{\dbar} \left( \sigmanuhatQ + \frac{\sqrt{\sigma_a^2+\sigma_b^2}}{\sqrt{n}} z_{\alpha/2} \right) \right).
	\end{align}
	
	In general, the values of $\sigma_a^2$, $\sigma_b^2$, and $\dbar$ are unknown quantities and need to be estimated.  The following proposition yields consistent estimators for each of those quantities.
	\begin{proposition}\label{propositionscalingestimation}
		Under the assumptions of Theorem \ref{theoremci},
		\begin{align*}
			&\sigmaahat \defined \frac{2n\sum_{k=1}^{n_a} \left( \sigmanuhat/\hat{d} + d_k^{-1}\sigmaepsilonhat \right)^2 }{\left(\Tr(D^{-1})\right)^2} \conp \sigma_a^2,\\
			&\sigmabhat \defined \frac{2n\left(\sigmaepsilonhat\right)^2}{n_b} \conp \sigma_b^2,\\
			&\hat{d} \defined \frac{n_a}{\Tr(D^{-1})} \conp \dbar.
		\end{align*}
	\end{proposition}
	
	Therefore, we may plug in the values of $\sigmaahat$, $\sigmabhat$, and $\dbar$ into equation (\ref{equationci}) to obtain an asymptotic confidence interval.
	
	\section{Empirical Bayes in ANOVA Type Models}\label{sectionestimation}
	
	The motivating example of this problem framework is in terms of the Rasch model, originally proposed by \citeA{rasch1960}.  The model that we consider is slightly different than the classical Rasch model in that we have Gaussian responses as opposed to binary responses.  
	
	As an example of this model, the data that we will consider in Section \ref{sectionrealdata} is from the Trends in Mathematics and Sciences Study (TIMSS), an international study conducted every four years to measure fourth and eighth grade student achievement in mathematics and science.  We will only consider data from the year 2015.  Countries randomly sample a collection of nationally representative schools to take standardized examinations in both mathematics and science, with questions being either multiple choice or constructed response.  Then, each student within schools takes only a subset of the questions on the exams but all questions are answered by some students in each school.  In addition to recording student responses, the data also contains background covariates for schools.  \citeA{martin2016} provides a more detailed description of the methods and procedures employed by TIMSS and more general information about TIMSS is available in \citeA{mullis20years}.
	
	For our analysis, we will only consider multiple choice questions and analyze on the level of school rather than students.  To construct a response variable for school, we compute the proportion of questions answered correctly by students in that school.  Note that, unlike the classical Rasch model, we assume a linear model and, for all schools, we have answers for all questions.  Thus, by the Central Limit Theorem, our response $Y$ will be approximately Gaussian.  The fixed effects design $X$ will be the background covariates for the school and the random effects design $Z$ will be an indicator for the country, with $\nu$ corresponding to the unobserved variability of the countries.  In this example, we do since we have averaged over questions, we do not have any nuisance random effects.  The problem that we will consider in this section is ranking the countries based on mathematical ability and trying to estimate the average number of questions that any particular country will answer correctly.  That is, we would like to estimate $\mu+Z\nu$ for all countries.
	
	\subsection{Model and Motivation}
	The general problem framework that we consider is for $K$-factor ANOVA models.  However, we will derive the results in the setting when $K=2$.  That is, we will consider the model
	\begin{align*}
		Y = \mu + Z\nu + W\gamma + \epsilon.
	\end{align*}
	We do not assume that the design is fully crossed in the random effects.  The goal in the problem is to estimate a subset of the mean vector, $\eta \defined \mu + Z\nu$, since we view the random effects $W$ as nuisance.  However, as the sample size increases, the number of observations per group stays bounded.  In the context of the motivating data example, each country still only answers a finite number of questions as we increase the sample size.  Therefore, it is not possible to consistently estimate the entire vector $\eta$.  A standard approach in the low-dimensional setting would be to use an empirical Bayes estimator by placing a Gaussian prior on both $\nu$ and $\gamma$ (for example, see \citeA{brown2018}), which would then transform the problem into a standard high-dimensional linear mixed model.  Therefore, we will use a $\n_v\left(0_v,\sigma_\nu^2 I_v\right)$ and $\n_r\left(0_r , \sigma_\gamma^2 I_r\right)$ prior on $\nu$ and $\gamma$ respectively.
	
	Since we will need to estimate both $\sigma_\nu^2$ and $\sigma_\gamma^2$ for the prior, our estimator for $\sigma_\gamma^2$ will be analogous to $\sigmanuhat$ from Section \ref{sectionci}.  To this end, we will need an additional matrix $C$ such that 
	\begin{align*}
		CY = CX\beta + CW\gamma + C\epsilon.
	\end{align*}
	Further, $C$ may be chosen such that $C$ has orthonormal rows that are mutually orthogonal with $(A^\T,B^\T)^\T$.  We will let $n_c$ denote the number of rows of $C$.  Similarly, we will need to consider the Spectral Decomposition of $CWW^\T C^\T = \Gamma \Lambda \Gamma^\T$.  Similarly, we have that $\Lambda$ is diagonal and $\Gamma$ is unitary.  Then, we may take the transformation
	\begin{align}\label{modelgamma}
		\Lambda^{-1/2} \Gamma^\T CY = \Lambda^{-1/2} \Gamma^\T X\beta + \Lambda^{-1/2} \Gamma^\T CW\gamma + \Lambda^{-1/2} \Gamma^\T C\epsilon
	\end{align}
	to whiten the data with respect to $\gamma$.
	
	\subsection{Estimator}
	To estimate $\mu$, we will define $\betahatew$ to be the exponentially weighted estimator for model (\ref{modellmmy}), with $\hat{\mu}_\text{EW} \defined X\betahatew$.
	
	From Theorem \ref{theoremci}, we immediately have that $\sigmanuhatQ \conp \dbar \sigma_\nu^2$.  This suggests the empirical prior $\n_v\left( 0_v , \sigmanuhatQ/\hat{d} \right)$ for $\nu$.  However, we will not use this estimator, but a slightly modified version by jointly estimating the vector $\mu$ using all of the observations.  Before we define a modified version of $\sigmanuhatQ$, we will first need to adapt our estimator of $\sigma_\epsilon^2$, which will be given by
	\begin{align*}
		\sigmaepsilonhat \defined \frac{1}{n_b} \left\Vert BY - BX\betahatew \right\Vert^2.
	\end{align*}
	Then, we may define
	\begin{align*}
		\sigmanuhat \defined \frac{1}{\Tr(D^{-1})}\left\Vert D^{-1/2}V^\T AY - D^{-1/2} V^\T AX\betahatew \right\Vert^2 - \sigmaepsilonhat.
	\end{align*}
	We may similarly define $\sigmagammahat$ as
	\begin{align*}
		\sigmagammahat \defined \frac{1}{\Tr(\Lambda^{-1})}\left\Vert \Lambda^{-1/2}\Gamma^\T CY - \Lambda^{-1/2} \Gamma^\T CX\betahatew \right\Vert^2 - \sigmaepsilonhat.
	\end{align*}
	Then, under some assumptions, it will follow that $\sigmagammahat \conp \lambdabar \sigma_\gamma^2$ analogous to Theorem \ref{theoremci} where $\lambdabar$ is defined in assumption \ref{assumptionvarianceconsistency}.  Hence, we will estimate $\sigma_\nu^2$ by $\sigmanuhat/\hat{d}$, $\sigma_\gamma^2$ by $\sigmagammahat/\hat{\lambda}$, and $\sigma_\epsilon^2$ by $\sigmaepsilonhat$.  This will give the following empirical Bayes estimator for $\eta$,
	\begin{align*}
		\etahat \defined \hat{\mu}_\text{EW} + \frac{\sigmanuhat}{\hat{d}} ZZ^\T \left( \frac{\sigmanuhat}{\hat{d}} ZZ^\T + \frac{\sigmagammahat}{\hat{\lambda}}WW^\T + \sigmaepsilonhat I_n \right)^{-1} \left( Y - \hat{\mu}_\text{EW} \right).
	\end{align*}

	To compare our estimator, we will consider an oracle that has access to $\mu$, $\sigma_\nu^2$, $\sigma_\gamma^2$, and $\sigma_\epsilon^2$.  Then, this oracle will use the Bayes estimator for $\eta$ (see Lemma \ref{lemmaeb}), given by
	\begin{align*}
		\etahatoracle \defined \mu + \sigma_\nu^2 ZZ^\T \left( \sigma_\nu^2 ZZ^\T + \sigma_\gamma^2 WW^\T + \sigma_\epsilon^2 I_n \right)^{-1} \left( Y - \mu \right).
	\end{align*}

	\subsection{Assumptions}
	Unlike the previous section, we will not need to establish the asymptotic distribution of $\sigmanuhat$, rather we only need the estimator to be consistent.  Accordingly, we may weaken our assumptions to the following
	
	\begin{enumerate}[label=(C\arabic*)]
		\item \label{assumptionvarianceconsistency} The matrices $D$ and $\Lambda$ satisfy
		\begin{align*}
			&\lim_{n\to\infty} \frac{\Tr\left( D^{-2} \right)}{\Tr \left( D^{-1} \right)^{2}} \to 0 &&& \lim_{n\to\infty} \frac{\Tr\left( \Lambda^{-2} \right)}{\Tr \left( \Lambda^{-1} \right)^{2}} \to 0\\
			&\lim_{n\to\infty} \frac{n_a}{\Tr\left( D^{-1} \right)} \to \dbar &&& \lim_{n\to\infty} \frac{n_c}{\Tr\left( \Lambda^{-1} \right)} \to \lambdabar.
		\end{align*}
		
		\item \label{assumptionsparsityeb} The sequence of constants $u$ satisfies
		\begin{align*}
		&\liminf_{n\to\infty} (u-s) \geq 0,\\
		&\lim_{n\to\infty} \frac{u\log(p)}{n_a \wedge n_b \wedge n_c} \to 0.
		\end{align*}
		
		\item \label{assumptionrandomeffectsizes} The design matrices satisfy
		\begin{align*}
			\left\Vert \sigma_\nu^2 ZZ^\T + \sigma_\gamma^2 WW^\T + \sigma_\epsilon^2 I_n \right\Vert^2 = \mathcal{O}(1).
		\end{align*}
	\end{enumerate}
	
	\begin{remark}
		Assumption \ref{assumptionvarianceconsistency} is weaker than the combination of assumptions \ref{assumptioneffectivesamplesize} and \ref{assumptionasymptoticsamplesize}.  This will ensure consistency in estimating $\sigma_\nu^2$, $\sigma_\gamma^2$, and $\sigma_\epsilon^2$.
		
		The next assumption, \ref{assumptionsparsityeb}, is identical to assumption \ref{assumptionsparsitynull} but with the additional constraint on $n_c$ to ensure consistent prediction of $\Lambda^{-1/2}\Gamma^\T C\mu$ in equation (\ref{modelgamma}).
		
		The last assumption, \ref{assumptionrandomeffectsizes}, is equivalent to, in the context of the motivating example, each school only answering $\mathcal{O}(1)$ number of questions and each question being answered by $\mathcal{O}(1)$ number of schools.
	\end{remark}

	\subsection{Main Results for $\etahat$}
	We will start this section by proving that $\sigmanuhat$, $\sigmagammahat$, and $\sigmaepsilonhat$ are all consistent estimators.
	\begin{proposition}\label{propositionvarianceconsistency}
		Consider the linear mixed model given by equation (\ref{modellmmy}), with the reduced models given in equations (\ref{modelyden}), (\ref{modelyci}), and (\ref{modelgamma}).  Then, assuming \ref{assumptionvarianceconsistency} and \ref{assumptionsparsityeb},
		\begin{align*}
			\sigmanuhat \conp \dbar \sigma_\nu^2,
			&&&\sigmagammahat \conp \lambdabar \sigma_\gamma^2,
			&&\sigmaepsilonhat \conp \sigma_\epsilon^2.
		\end{align*}
	\end{proposition}
	
	The following is a standard lemma regarding the empirical Bayes estimators in this problem setup, which we will prove for the sake of completeness.
	\begin{lemma}\label{lemmaeb}
		For a fixed vector $\mu \in \R^n$ and fixed values $\sigma_\nu^2>0$, $\sigma_\gamma^2>0$, and $\sigma_\epsilon^2 >0$, the Bayes estimator of $\eta$ is given by
		\begin{align*}
			\e\left( \eta | Y \right) = \mu + \sigma_\nu^2 ZZ^\T \left( \sigma_\nu^2 ZZ^\T + \sigma_\gamma^2 WW^\T + \sigma_\epsilon^2 I_n \right)^{-1} \left( Y - \mu \right).
		\end{align*}
	\end{lemma}
	
	We conclude this section with the main result regarding $\etahat$, that the empirical Bayes estimator performs nearly as well as the oracle Bayes estimator $\etahatoracle$ asymptotically.
	\begin{theorem}\label{theoremeboracle}
		Consider the linear mixed model given in (\ref{modellmmy}).  Assume \ref{assumptionvarianceconsistency}, \ref{assumptionsparsityeb}, and \ref{assumptionrandomeffectsizes}.  Then,
		\begin{align*}
			\frac{1}{n} \left( \left\Vert \etahat - \eta \right\Vert^2 - \left\Vert \etahatoracle - \eta \right\Vert^2 \right) = \op(1).
		\end{align*}
	\end{theorem}
	
	\section{Implementation and Simulations}\label{sectionnumerics}
	\subsection{Implementation}
	
	In this section, we describe how to compute all three of the proposed estimators.  The computation can be split into two halves:
	\begin{enumerate}
		\item Obtaining the matrices $A$, $B$, $C$, $D$, $V$, $\Lambda$, and $\Gamma$ with which to project the data.
		\item Computing the estimators via exponential weighting.
	\end{enumerate}
	
	For the first part, the matrices $A$, $B$, and $C$ can be obtained by adding the $n$ standard basis vectors in $\R^n$ to $(W,Z)$ to form a basis for $\R^n$ and applying the modified Gram-Schmidt procedure.  To compute $D$ and $V$, we form the matrices $AZZ^\T A^\T$ and directly apply the eigendecomposition to obtain both the eigenvalues and the eigenvectors.  An analogous procedure is used to compute $\Lambda$ and $\Gamma$.
	
	For the second part, it may be further sub-divided into three parts:
	\begin{enumerate}
		\item Determining a bound $\alpha$.
		\item Determining a value of $u$.
		\item Computing $|\m_u| = {p \choose u} = \mathcal{O}(p^u)$ models.
	\end{enumerate}
	To determine $\alpha$, we only need an upper bound on the maximal eigenvalue of the noise vector.  This may simply be bounded by the variance of $Y$ by choosing $\alpha = 4 \left\Vert Y \right\Vert^2/n$.  Though this will choose a conservative value of $\alpha$, asymptotically this will have no effect.  For $u$, we may following the suggestion of \citeA{rigollet2011} and apply the exponential screening prior.  They say that a covariate $X_j$ is \emph{selected} if $\hat{\beta}_\text{ES,j} > 1/n$, where $\hat{\beta}_\text{ES,j}$ is the $j$'th entry of the exponential screening estimator.  Note that, after computing the matrix $B$, the model $BY = BX\beta + B\epsilon$ is a standard high-dimensional linear model with the same level of sparsity as the original model, which is in the framework of \citeA{rigollet2011}.  Alternatively, we may apply cross-validation to tune the sparsity level, though this is very computationally intensive.
	
	Once both are selected, we still need to compute ${p \choose u}$ least-squares estimators.  However, \citeA{rigollet2011} and \citeA{rigollet2012} proposed a Metropolis Hastings approach to approximating the exponential weighting estimator.  We will refer the interested reader to \citeA{rigollet2012} for the details.

	\subsection{Simulations}
	\subsubsection{Methods and Models}
	
	We will consider the linear mixed model given by
	\begin{align*}
		Y = X\beta + Z\nu + W\gamma + \epsilon,
	\end{align*}
	with $n = 200$, $s=3$, and $p=500$.  For simplicity, we will set $\mu = X\beta$.  Since no assumption is made on the fixed effects design $X$, we will generate $X$ from a multivariate Gaussian $\n_n\left(0_n, \Sigma\right)$, where $\Sigma$ is an equi-correlation matrix.  That is,
	\begin{align*}
		\Sigma_{i,j} = 
		\begin{cases}
			1 \text{ if } i=j\\
			\rho \text{ if } i\neq j
		\end{cases}
	\end{align*}
	for $\rho \in \{0,0.8\}$.  As our fixed effects design is perfectly symmetric, we will take the strongly sparse set $S$ to be the first $s$ components with $\beta_j = 1$ for $j\in S$.  We will fix the value of $\sigma_\epsilon^2=1$.  For simplicity, we will let $\nu$ and $\gamma$ each correspond to a single random effect that are independent of each other.  Therefore, we will have that
	\begin{align*}
		\begin{pmatrix}
			\nu \\ \gamma \\ \epsilon
		\end{pmatrix} \sim 
		\n_{v+r+n} \left( 
		\begin{pmatrix}
			0_v \\ 0_r \\ 0_n
		\end{pmatrix} , 
		\begin{pmatrix}
			\sigma_\nu^2 I_v & 0_{v\times r} & 0_{v\times n}\\
			0_{r\times v} & \sigma_\gamma^2 I_r & 0_{r\times n}\\
			0_{n\times v} & 0_{n\times r} & \sigma_\epsilon^2 I_n
		\end{pmatrix}
		\right).
	\end{align*}
	Finally, with regards to the random effects design, we will consider the setting where $(v,r)\in \{25,50\}^2$.  Since the number of observations will be less than the number of possible combinations, we will start by generating a fully crossed design and then down sample $n$ rows without replacement of the fully crossed design.  To experiment with both the level and the power of our procedure, we will consider $(\sigma_\nu^2,\sigma_\gamma^2) \in \{0,1\}^2$.  Hence, we will consider each of the following settings
	\begin{align*}
		\left(\rho,v,r,\sigma_\nu^2,\sigma_\gamma^2\right) \in \{0,0.8\}\times \{25,50\} \times \{25,50\} \times \{0,1\} \times \{0,1\}.
	\end{align*}
	
	For each setting, we will generate a single draw of the design for $(X,W,Z)$.  Then, we will draw $N=100$ trials of $(\nu,\gamma,\epsilon)$ and compute each of our estimators on these trials.  Whenever a variance estimate is negative, we will set the estimated value to zero.  When using exponential weighting, we will estimate the sparsity using the exponential screening prior of \citeA{rigollet2011} on the model $BY = BX\beta + B\epsilon$.  Moreover, we will use the value of $\alpha = 4\left\Vert Y \right\Vert^2/n$.
	
	All of our simulations are conducted in R.  For each of our three problems, we will compare our estimator with an oracle low-dimensional estimator as well as a low-dimensional version of our proposed high-dimensional statistic.  For the oracle estimators, in the setting of the $F$-test, we directly apply the classical low-dimensional $F$-test that has access to the true sparse set $S$, as given in equation (2.3) of \citeA{jiang2007}, which we will denote by $\fld$.  For the confidence intervals, we fit the linear mixed models with the true sparse set $S$ using {\tt lmer} and applying the {\tt confint} function.  We will denote this estimator as $\sigmanuhatld$.  Finally, in the setting of estimation, we directly compute the oracle Bayes estimator $\etahatoracle$ described in Section \ref{sectionestimation}.  Collectively, things subscripted with ``LD'' will denote low-dimensional estimators.
	
	In addition to comparing with the low-dimensional estimators, we also construct low-dimensional versions of our proposed high-dimensional statistics.  To do so, we use the exact same statistic as in the high-dimensional setting but replace exponential weighting with least-squares using the sparse set $S$.  Hence, all of these statistics will be subscripted by ``LS'' to denote the usage of least-squares.  We make this comparison since all of our proposed statistics rely on two layers of asymptotics:  
	\begin{enumerate}
		\item In the prediction of the mean vector via exponential weighting.
		\item In the convergence once the residuals are obtained.
	\end{enumerate}
	To differentiate between these two, we introduce an intermediate statistic that relies on least-squares, which we think of as low-dimensional versions of our statistics.  For example, recalling that $\hat{\beta}_S$ is the least-squares estimator of $\beta$ using the covariates $X_S$, we will also consider the statistic
	\begin{align*}
		\fls \defined \frac{\left\Vert AY - AX\hat{\beta}_{S} \right\Vert^2/{n_a}}{\left\Vert BY - BX\hat{\beta}_{S} \right\Vert^2/{n_b}} \sim F_{n_a,n_b} + \op(1).
	\end{align*}
	Similarly, we will consider $\sigmanuhatls$ and $\sigmaepsilonhatls$ given by
	\begin{align*}
		&\sigmaepsilonhatls \defined \frac{1}{n_b} \left\Vert B\epsilon - BX\hat{\beta}_{S} \right\Vert^2,\\
		&\sigmanuhatls \defined \frac{1}{\Tr(D^{-1})}\left\Vert D^{-1/2}V^\T AY - D^{-1/2} V^\T AX\hat{\beta}_S \right\Vert^2 - \sigmaepsilonhatls,\\
		&\sigmagammahatls \defined \frac{1}{\Tr(\Lambda^{-1})}\left\Vert \Lambda^{-1/2}\Gamma^\T CY - \Lambda^{-1/2} \Gamma^\T CX\hat{\beta}_S \right\Vert^2 - \sigmaepsilonhatls.
	\end{align*}
	Then, we will have that
	\begin{align*}
		\sqrt{n} \left( \sigmanuhatls - \dbar\sigma_\nu^2 \right) \cond \n(0,\sigma_a^2 + \sigma_b^2)
	\end{align*}
	and
	\begin{align*}
		\etahatls \defined X\hat{\beta}_S + \frac{\sigmanuhatls}{\hat{d}} ZZ^\T \left( \frac{\sigmanuhatls}{\hat{d}} ZZ^\T + \frac{\sigmagammahatls}{\hat{\lambda}}WW^\T + \sigmaepsilonhatls I_n \right)^{-1} \left( Y - X\hat{\beta}_S \right).
	\end{align*}
	To compare the three procedures, we will consider the following metrics
	\begin{enumerate}
		\item Average Coverage:  The percentage of time the correct hypothesis is selected for $F$-tests or the percentage of time the true value of $\sigma_\nu^2$ is in the confidence interval.
		\item Average Length:  The average length of the confidence interval, taken as the upper endpoint minus the lower endpoint.
		\item Average Loss:  The average squared Euclidean distance between the estimated vector $\hat{\eta}$ and the true vector $\eta$ divided by $n$.
	\end{enumerate}
	
	\subsubsection{Results}
	Tables (\ref{tablesimulationgimanu0sigmagamma0}) - (\ref{tablesimulationgimanu1sigmagamma1}) provide the results of our simulations.  We notice that, on average, $\few$ and $\sigmanuhat$ are slightly less conservative than $\fld$ and $\sigmanuhatld$ respectively.  However, as $v$ increases from $25$ to $50$, we generally notice that type I error decreases and coverage increases, suggesting that the main problem that is asymptotic.  In particular, even with the correct low-dimensional covariates, $\fls$ and $\sigmanuhatls$ still have higher type I error and lower coverage but does outperform both $\few$ and $\sigmanuhat$, as expected.  However, the confidence intervals for $\sigmanuhat$ seem to be overly optimistic when the design is more correlated and the true value of $\sigma_\nu^2=1$.  In regards to estimation of $\eta$, the loss for $\etahat$ is generally close to the oracle loss, as predicted by Theorem \ref{theoremeboracle}.
	
	Moreover, Figures (\ref{plotsimulationsigmanu0sigmagamma0}) and (\ref{plotsimulationsigmanu1sigmagamma0}) show some simulated confidence intervals for $10$ trials.  The fact that $\fls$ and $\few$ occasionally have shorter confidence intervals than profile likelihood might be attributable to the more optimistic nature of those two procedures.
	
	\begin{table}
		\centering
		\caption{Simulations when $\sigma_\nu^2 = 0$ and $\sigma_\gamma^2 =0$}
		\label{tablesimulationgimanu0sigmagamma0}
		\begin{tabular}{|l|l|rrrrrrrr|}
			\hline
			& $\rho$ & 0 & 0 & 0 & 0 & 0.8 & 0.8 & 0.8 & 0.8 \\ \hline
			& $v$ & 25 & 25 & 50 & 50 & 25 & 25 & 50 & 50 \\ \hline
			& $r$ & 25 & 50 & 25 & 50 & 25 & 50 & 25 & 50 \\ \hline
			& $\fld$ & 92 & 91 & 94 & 98 & 95 & 95 & 99 & 93 \\ 
			AveCov & $\fls$ & 92 & 91 & 95 & 98 & 96 & 95 & 97 & 94 \\ 
			& $\few$ & 91 & 92 & 95 & 84 & 93 & 100 & 82 & 99 \\  \hline
			& $\sigmanuhatld$ & 97 & 99 & 99 & 100 & 98 & 96 & 99 & 97 \\ 
			AveCov & $\sigmanuhatls$ & 100 & 100 & 98 & 98 & 100 & 100 & 99 & 99 \\ 
			& $\sigmanuhat$ & 100 & 100 & 100 & 99 & 100 & 100 & 97 & 100 \\  \hline
			& $\sigmanuhatld$ & 0.33 & 0.32 & 0.31 & 0.34 & 0.33 & 0.31 & 0.36 & 0.35 \\  
			AveLen & $\sigmanuhatls$ & 0.12 & 0.16 & 0.26 & 0.43 & 0.13 & 0.15 & 0.25 & 0.35 \\  
			& $\sigmanuhat$ & 0.12 & 0.17 & 0.29 & 0.84 & 0.12 & 0.13 & 0.26 & 0.34 \\ \hline
			& $\etahatoracle$ & 0 & 0 & 0 & 0 & 0 & 0 & 0 & 0 \\ 
			AveLoss & $\etahatls$ & 0.02 & 0.02 & 0.03 & 0.03 & 0.02 & 0.02 & 0.02 & 0.03 \\ 
			& $\etahat$  & 0.07 & 0.15 & 0.23 & 1.00 & 0.32 & 0.32 & 0.32 & 0.39 \\ 
			\hline
		\end{tabular}
	\end{table}

	\begin{table}
		\centering
		\caption{Simulations when $\sigma_\nu^2 = 0$ and $\sigma_\gamma^2 =1$}
		\label{tablesimulationgimanu0sigmagamma1}
		\begin{tabular}{|l|l|rrrrrrrr|}
			\hline
			& $\rho$ & 0 & 0 & 0 & 0 & 0.8 & 0.8 & 0.8 & 0.8 \\ \hline
			& $v$ & 25 & 25 & 50 & 50 & 25 & 25 & 50 & 50 \\ \hline
			& $r$ & 25 & 50 & 25 & 50 & 25 & 50 & 25 & 50 \\ \hline
			& $\fld$ & 95 & 96 & 96 & 94 & 94 & 93 & 95 & 97 \\ 
			AveCov & $\fls$ & 96 & 96 & 97 & 94 & 95 & 92 & 97 & 98 \\ 
			& $\few$ & 95 & 98 & 95 & 97 & 93 & 99 & 99 & 99 \\  \hline
			& $\sigmanuhatld$ & 97 & 99 & 100 & 96 & 97 & 96 & 96 & 99 \\ 
			AveCov & $\sigmanuhatls$ & 100 & 100 & 100 & 96 & 100 & 100 & 100 & 100 \\ 
			& $\sigmanuhat$ & 100 & 100 & 99 & 98 & 100 & 100 & 99 & 99 \\  \hline
			& $\sigmanuhatld$ & 0.34 & 0.35 & 0.4 & 0.4 & 0.34 & 0.38 & 0.39 & 0.39 \\ 
			AveLen & $\sigmanuhatls$ & 0.20 & 0.16 & 0.22 & 0.32 & 0.12 & 0.16 & 0.30 & 0.36 \\ 
			& $\sigmanuhat$ & 0.19 & 0.17 & 0.22 & 0.55 & 0.12 & 0.14 & 0.28 & 0.35 \\  \hline
			& $\etahatoracle$ & 0 & 0 & 0 & 0 & 0 & 0 & 0 & 0 \\ 
			AveLoss & $\etahatls$ & 0.04 & 0.03 & 0.04 & 0.04 & 0.04 & 0.04 & 0.04 & 0.05 \\ 
			& $\etahat$ & 0.12 & 0.28 & 0.19 & 1.07 & 0.41 & 0.39 & 0.39 & 0.38 \\ 
			\hline
		\end{tabular}
	\end{table}
	
	\begin{table}
		\centering
		\caption{Simulations when $\sigma_\nu^2 = 1$ and $\sigma_\gamma^2 =0$}
		\label{tablesimulationgimanu1sigmagamma0}
		\begin{tabular}{|l|l|rrrrrrrr|}
			\hline
			& $\rho$ & 0 & 0 & 0 & 0 & 0.8 & 0.8 & 0.8 & 0.8 \\ \hline
			& $v$ & 25 & 25 & 50 & 50 & 25 & 25 & 50 & 50 \\ \hline
			& $r$ & 25 & 50 & 25 & 50 & 25 & 50 & 25 & 50 \\ \hline
			& $\fld$ & 100 & 100 & 100 & 100 & 100 & 100 & 100 & 100 \\ 
			AveCov & $\fls$ & 100 & 100 & 100 & 100 & 100 & 100 & 100 & 100 \\ 
			& $\few$ & 100 & 100 & 98 & 97 & 100 & 100 & 100 & 100 \\  \hline
			& $\sigmanuhatld$ & 96 & 95 & 96 & 95 & 95 & 94 & 96 & 96 \\ 
			AveCov & $\sigmanuhatls$ & 87 & 92 & 93 & 91 & 91 & 85 & 94 & 96 \\ 
			& $\sigmanuhat$ & 80 & 87 & 89 & 95 & 62 & 42 & 51 & 73 \\ \hline
			& $\sigmanuhatld$ & 0.64 & 0.66 & 0.52 & 0.52 & 0.67 & 0.63 & 0.52 & 0.53 \\ 
			AveLen & $\sigmanuhatls$ & 1.22 & 1.34 & 1.12 & 1.31 & 1.32 & 1.19 & 1.05 & 1.27 \\ 
			& $\sigmanuhat$ & 1.10 & 1.26 & 1.18 & 1.70 & 0.92 & 0.83 & 0.82 & 1.02 \\  \hline
			& $\etahatoracle$ & 0.11 & 0.12 & 0.19 & 0.19 & 0.11 & 0.11 & 0.19 & 0.19 \\ 
			AveLoss & $\etahatls$ & 0.14 & 0.14 & 0.22 & 0.22 & 0.14 & 0.13 & 0.22 & 0.21 \\ 
			& $\etahat$ & 0.25 & 0.54 & 0.70 & 1.07 & 0.41 & 0.41 & 0.51 & 0.44 \\ 
			\hline
		\end{tabular}
	\end{table}
	
	\begin{table}
		\centering
		\caption{Simulations when $\sigma_\nu^2 = 1$ and $\sigma_\gamma^2 =1$}
		\label{tablesimulationgimanu1sigmagamma1}
		\begin{tabular}{|l|l|rrrrrrrr|}
			\hline
			& $\rho$ & 0 & 0 & 0 & 0 & 0.8 & 0.8 & 0.8 & 0.8 \\ \hline
			& $v$ & 25 & 25 & 50 & 50 & 25 & 25 & 50 & 50 \\ \hline
			& $r$ & 25 & 50 & 25 & 50 & 25 & 50 & 25 & 50 \\ \hline
			& $\fld$ & 100 & 100 & 100 & 100 & 100 & 100 & 100 & 100 \\ 
			AveCov & $\fls$ & 100 & 100 & 100 & 100 & 100 & 100 & 100 & 99 \\ 
			& $\few$ & 99 & 100 & 100 & 94 & 100 & 100 & 100 & 98 \\ \hline
			& $\sigmanuhatld$ & 93 & 98 & 98 & 95 & 97 & 96 & 95 & 96 \\ 
			AveCov & $\sigmanuhatls$ & 89 & 95 & 95 & 94 & 87 & 88 & 93 & 88 \\ 
			& $\sigmanuhat$ & 81 & 82 & 87 & 79 & 59 & 43 & 66 & 57 \\ \hline
			& $\sigmanuhatld$ & 0.68 & 0.68 & 0.69 & 0.55 & 0.68 & 0.68 & 0.56 & 0.57 \\ 
			AveLen & $\sigmanuhatls$ & 1.31 & 1.32 & 1.15 & 1.12 & 1.22 & 1.28 & 1.13 & 1.19 \\ 
			& $\sigmanuhat$ & 1.20 & 1.17 & 1.10 & 1.20 & 0.88 & 0.87 & 0.89 & 0.93 \\  \hline
			& $\etahatoracle$ & 0.13 & 0.14 & 0.22 & 0.24 & 0.14 & 0.14 & 0.22 & 0.23 \\ 
			AveLoss & $\etahatls$ & 0.18 & 0.18 & 0.27 & 0.28 & 0.19 & 0.19 & 0.26 & 0.28 \\ 
			& $\etahat$ & 0.39 & 0.41 & 0.46 & 0.97 & 0.50 & 0.52 & 0.52 & 0.59 \\ 
			\hline
		\end{tabular}
	\end{table}
	
	\begin{figure}
		\caption{Plots of the estimated confidence intervals in the first $10$ simulations when $\sigma_\nu^2 =0$ and $\sigma_\gamma^2 = 0$}
		\label{plotsimulationsigmanu0sigmagamma0}
		\begin{subfigure}{.49\textwidth}
			\centering
			\includegraphics[scale=0.275]{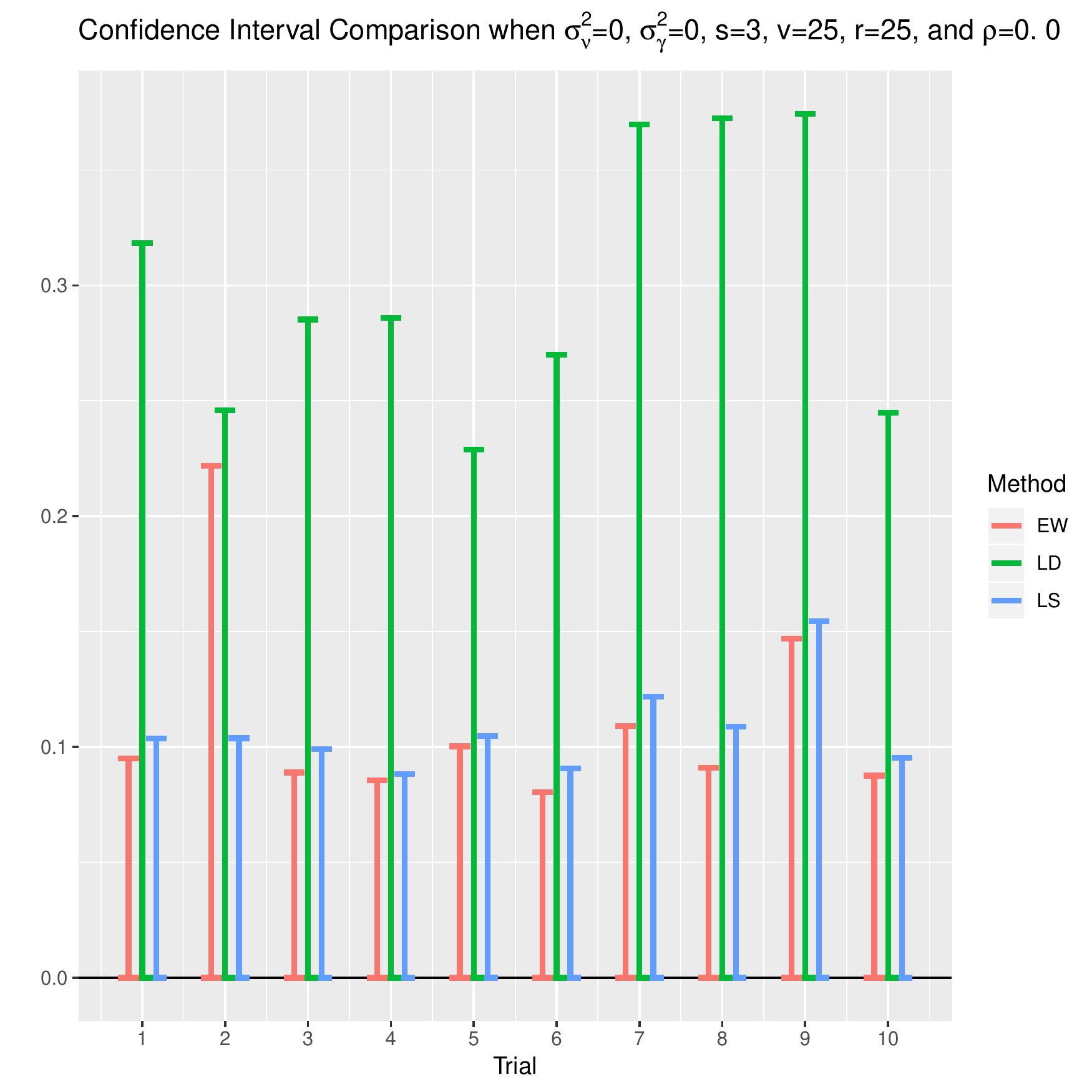}
			\caption{$\rho = 0$, $v=25$, $r=25$}
		\end{subfigure}
		\begin{subfigure}{.49\textwidth}
			\centering
			\includegraphics[scale=0.275]{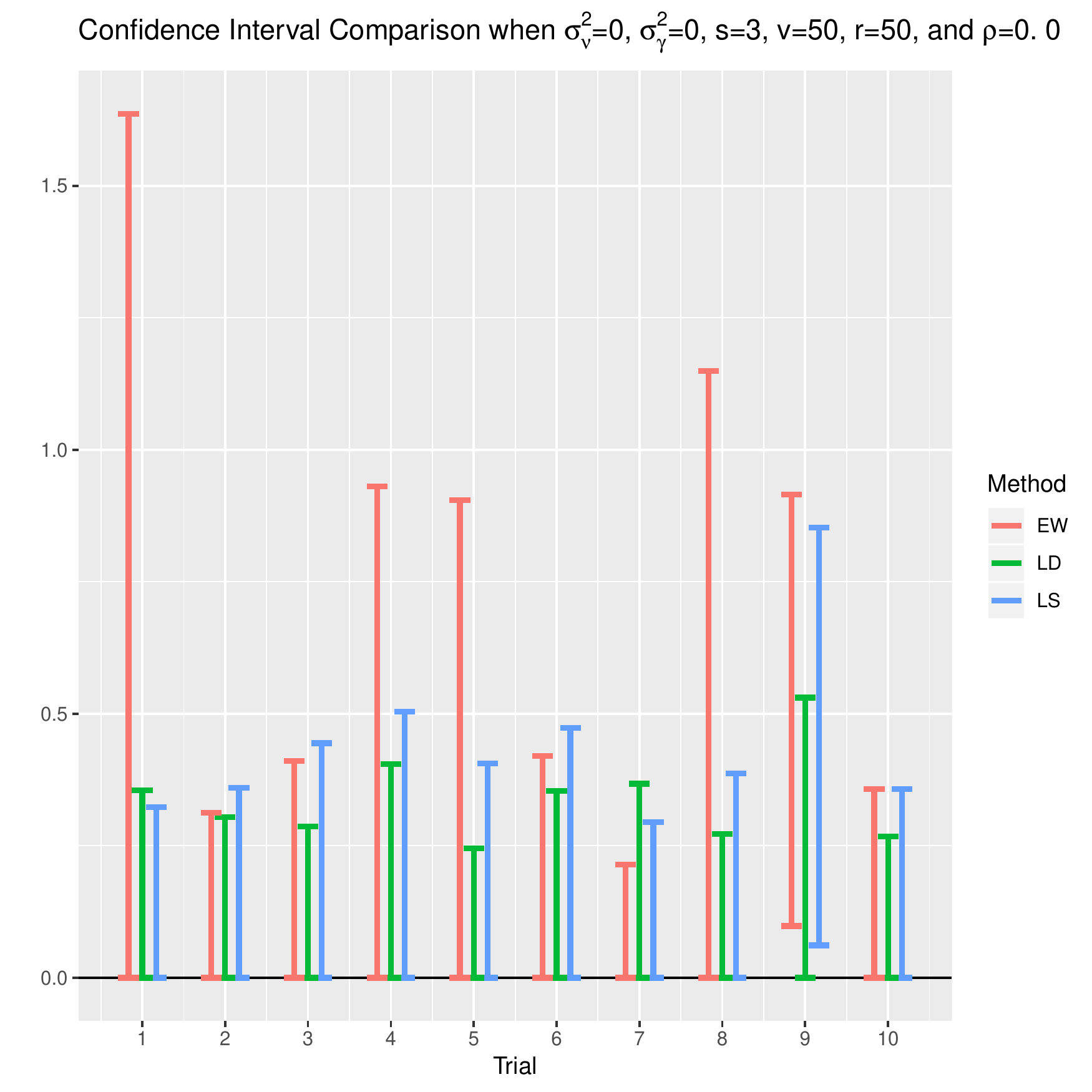}
			\caption{$\rho = 0$, $v=50$, $r=50$}
		\end{subfigure}
		\begin{subfigure}{.49\textwidth}
			\centering
			\includegraphics[scale=0.275]{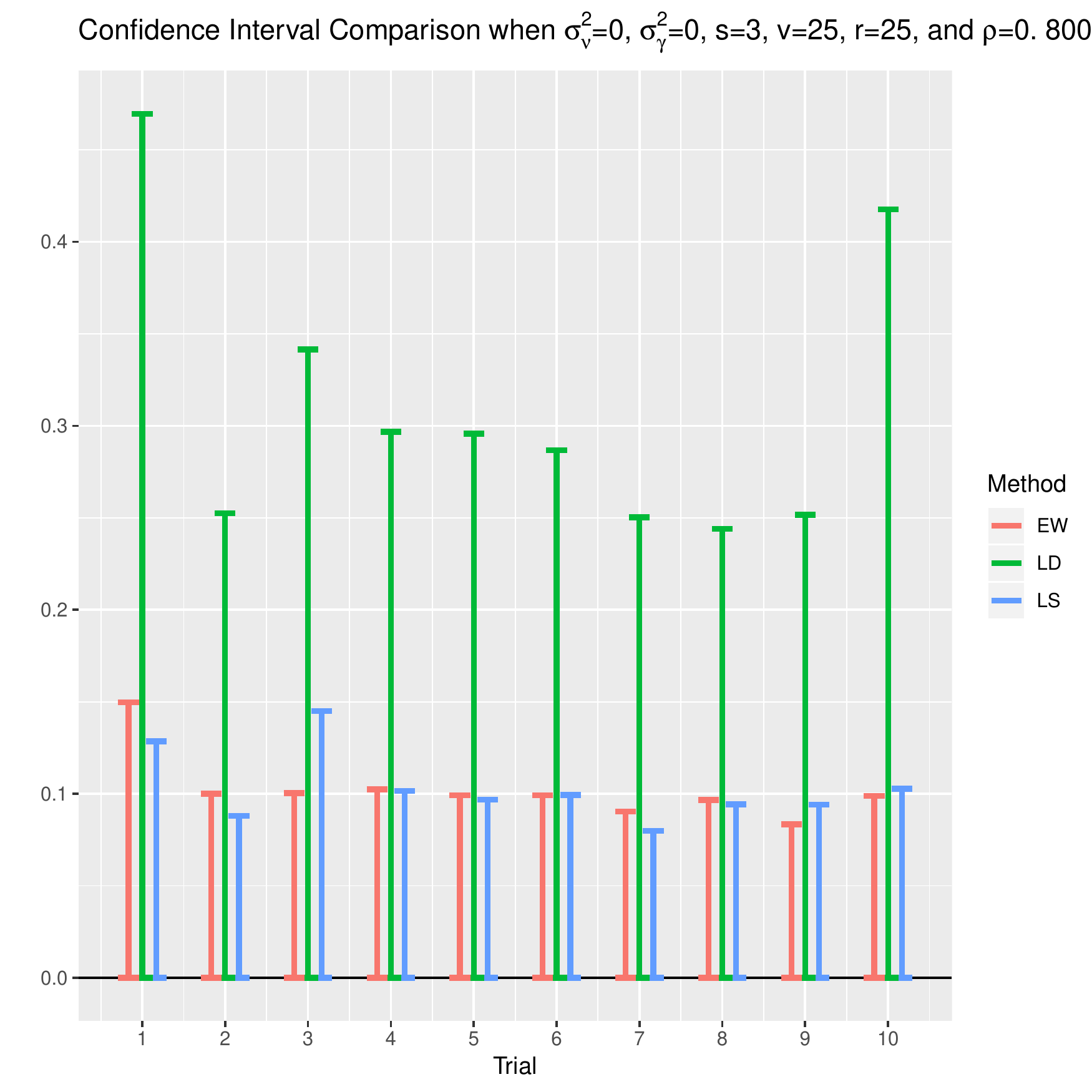}
			\caption{$\rho = 0.800$, $v=25$, $r=25$}
		\end{subfigure}
		\begin{subfigure}{.49\textwidth}
			\centering
			\includegraphics[scale=0.275]{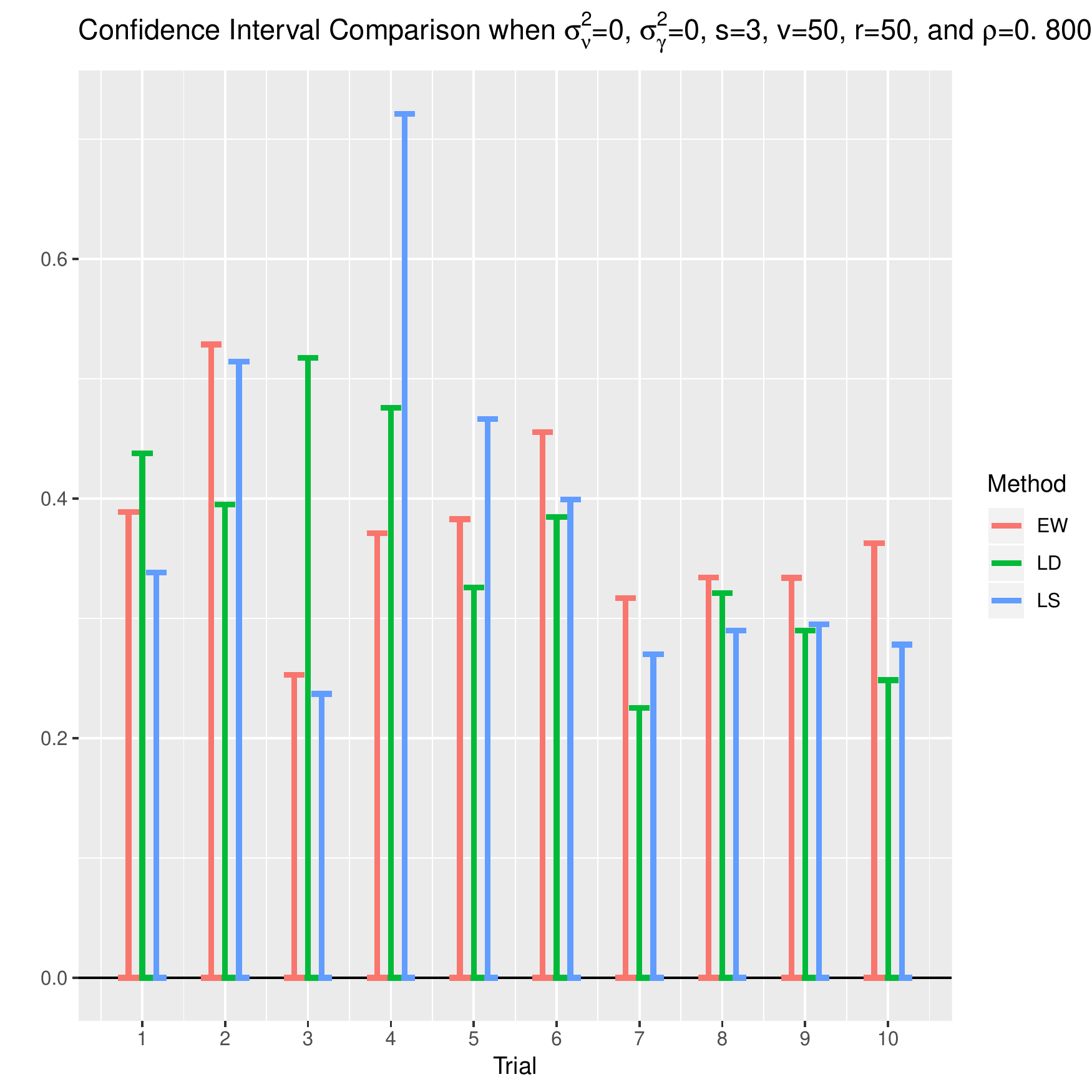}
			\caption{$\rho = 0.800$, $v=50$, $r=50$}
		\end{subfigure}
	\end{figure}	
	
	\begin{figure}
		\caption{Plots of the estimated confidence intervals in the first $10$ simulations when $\sigma_\nu^2 =1$ and $\sigma_\gamma^2 = 0$}
		\label{plotsimulationsigmanu1sigmagamma0}
		\begin{subfigure}{.49\textwidth}
			\centering
			\includegraphics[scale=0.275]{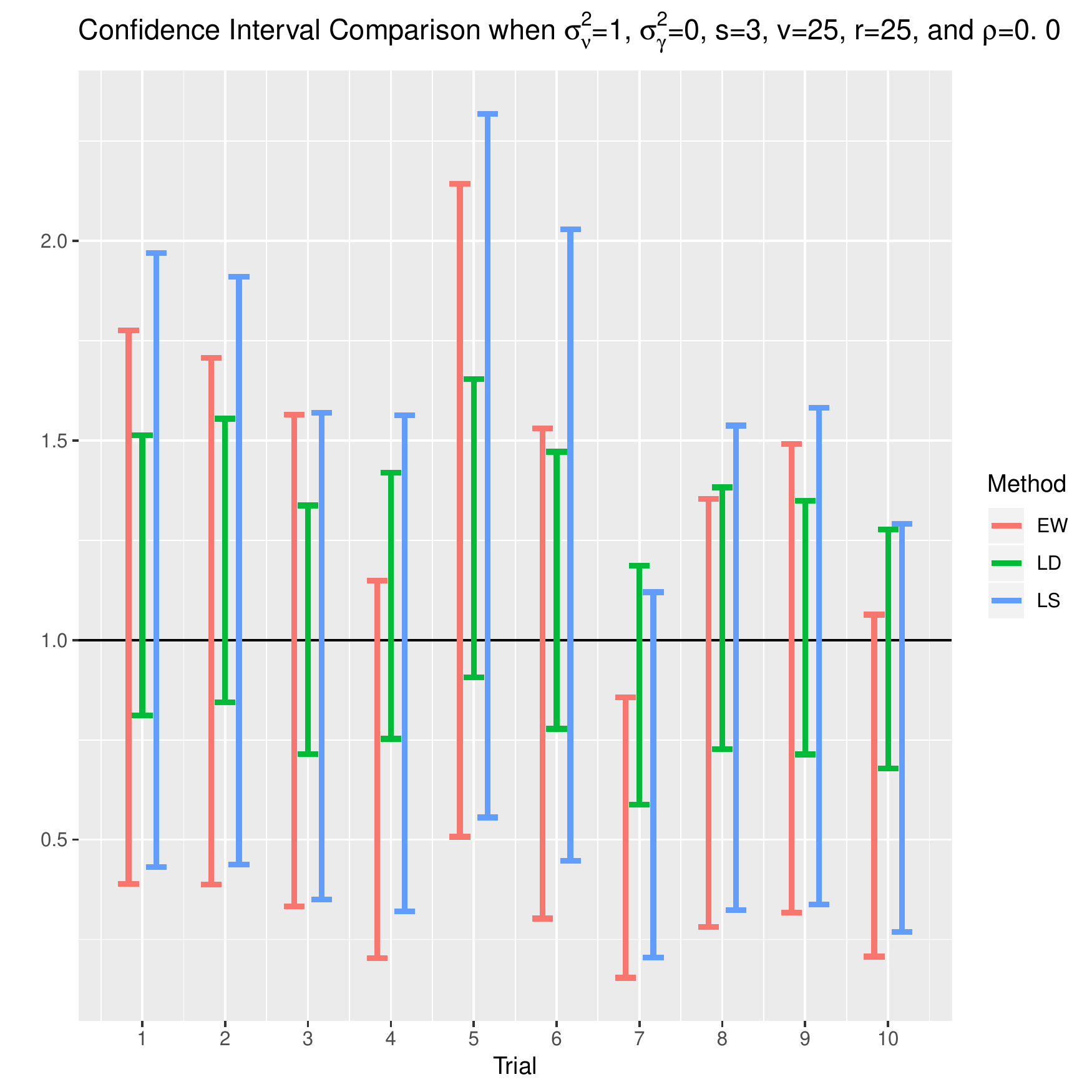}
			\caption{$\rho = 0$, $v=25$, $r=25$}
		\end{subfigure}
		\begin{subfigure}{.49\textwidth}
			\centering
			\includegraphics[scale=0.275]{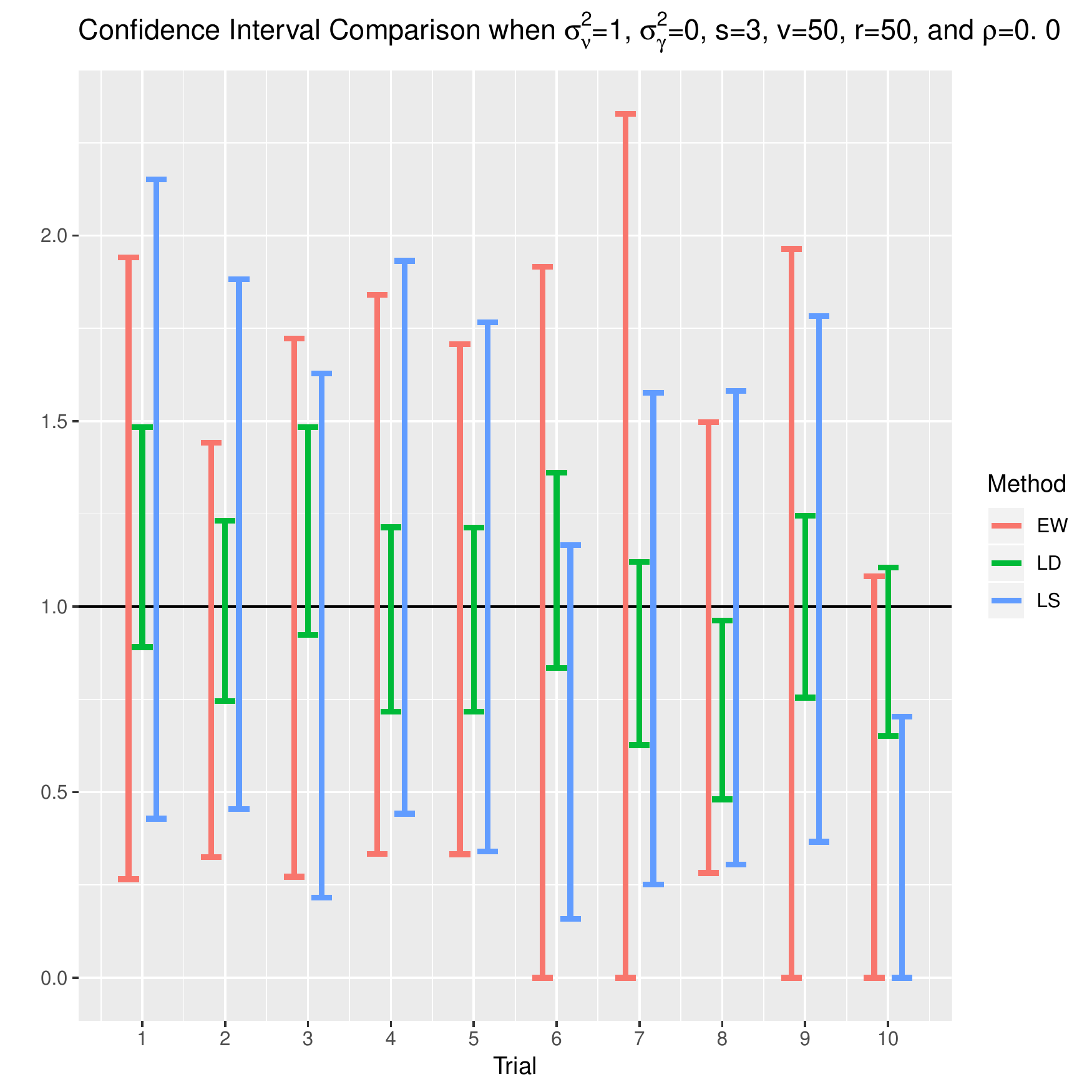}
			\caption{$\rho = 0$, $v=50$, $r=50$}
		\end{subfigure}
		\begin{subfigure}{.49\textwidth}
			\centering
			\includegraphics[scale=0.275]{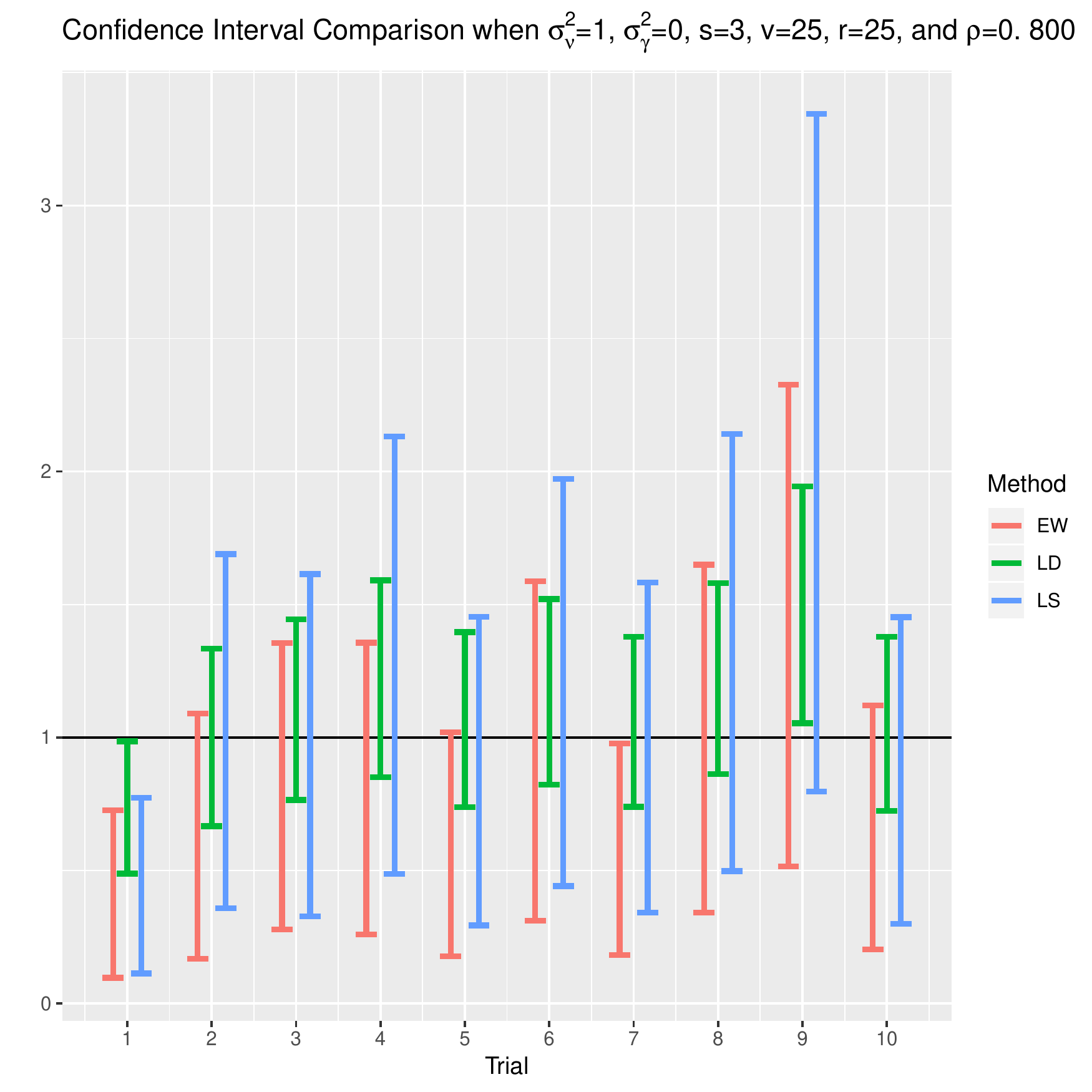}
			\caption{$\rho = 0.800$, $v=25$, $r=25$}
		\end{subfigure}
		\begin{subfigure}{.49\textwidth}
			\centering
			\includegraphics[scale=0.275]{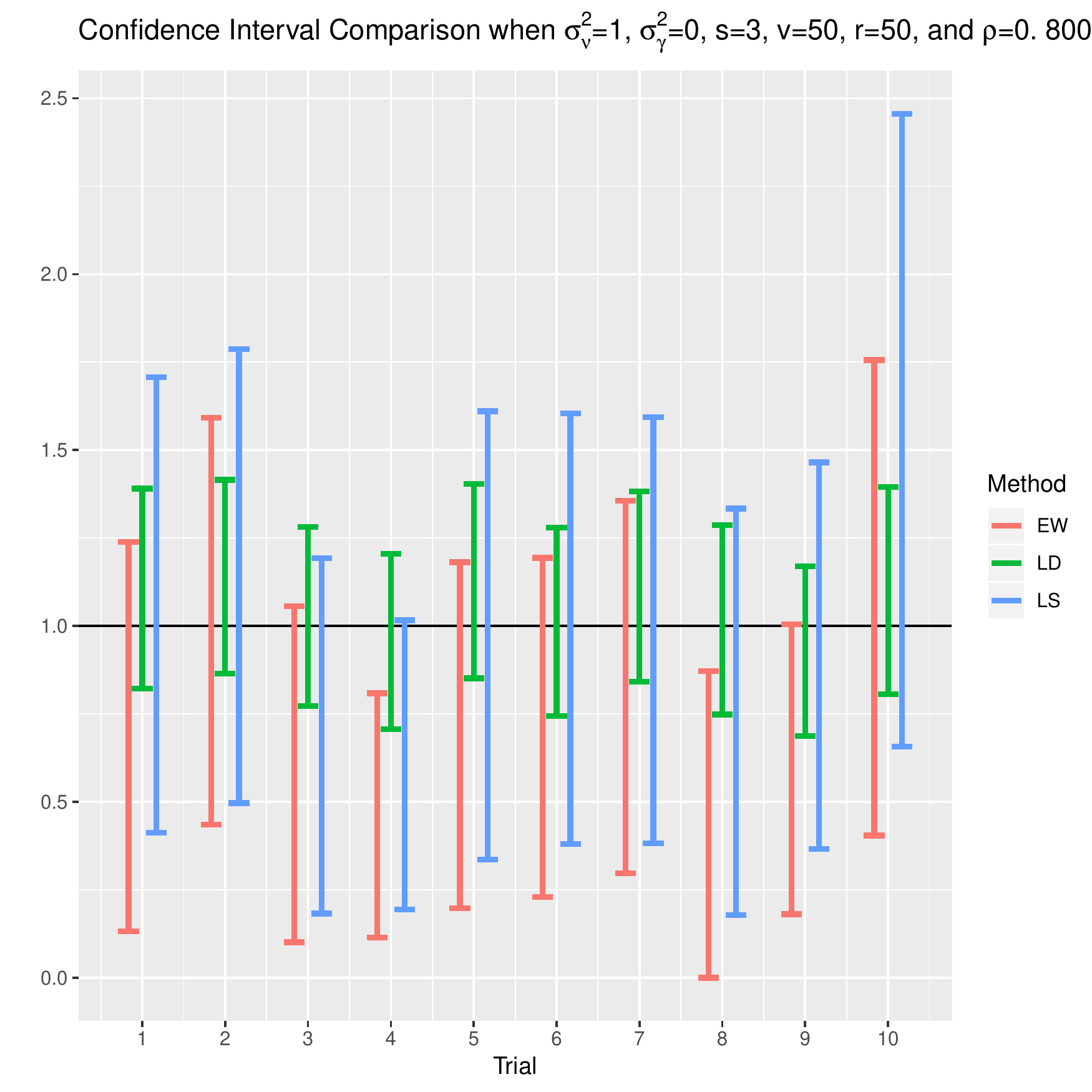}
			\caption{$\rho = 0.800$, $v=50$, $r=50$}
		\end{subfigure}
	\end{figure}
	
	\subsection{Real Data Application}\label{sectionrealdata}
	
	Following in the motivating example of Section \ref{sectionestimation}, we consider the Trends in International Mathematics and Sciences Study dataset, which is freely available online.  To simplify our analysis, we only consider the mathematics questions.  After filtering out for complete cases on background covariates, we are left with $146$ questions, $r=43$ unique countries, $p=106$ covariates, and $6808$ schools.  Therefore, we had a total of $n=6808$ responses after averaging over the students and questions within the schools.  
	
	To demonstrate our methodology, we apply all three procedures on the data.  When applying our methodology, we tune the sparsity using the exponential screening prior of \citeA{rigollet2011} and set $\alpha = 4\left\Vert Y \right\Vert^2/n$.  The high-dimensional $F$-test rejected the null hypothesis that $\sigma_\nu^2=0$ and a $95\%$ confidence interval for $\sigma_\nu^2$ is $(0.0020,0.0051)$, which suggests that, even controlling for school background characteristics, the country of the school impacts mathematical ability.  For the last part, we define a country's background characteristics $X$ to be the arithmetic average of all the schools' background characteristics within that country.  Then, applying the empirical Bayes procedure, we rank the countries based on the predicted number of questions they would answer correctly.  The top five countries in order from our analysis are South Korea, Hong Kong, Singapore, Chinese Taipei, and Japan.  Up to some reordering, our results are mostly consistent with the report of \citeA{mullis2016} based on individual student data, who had the same top five countries.

	\section{Discussions}\label{sectiondiscussion}
	In this paper, we considered three problems related to high-dimensional linear mixed models.  Throughout, we used exponential weighting to perform high-dimensional prediction due to its nice theoretical properties that do not require any assumption on the design matrix.  However, exponential weighting may be replaced with another procedure as long as we also satisfy the conditions for an oracle inequality for prediction, such as assuming compatibility for the lasso.  The theory for all three problems will remain unchanged.  In particular, the lasso has the advantage of being a convex optimization problem which is easily implemented by standard software.  
	
	The main drawback to our approach is the assumption that both the random effects and the error distributions are Gaussian, which is used to maintain independence after constructing various projections.  However, there are many situations when this is not a reasonable assumption and other procedures will need to be developed.  We view our contribution as a step towards these problems, similar to the development of the theory for linear mixed models in the low-dimensional setting.
	
	\section{Proofs}\label{sectionproofs}
	\subsection{Proofs for Section \ref{sectiontesting}}
	
	\begin{proof}[Proof of Theorem \ref{theoremnullhypothesis}]
		Under the null hypothesis, we have that
		\begin{align*}
			QY = Q\mu + Q\epsilon,
		\end{align*}
		which is a standard high-dimensional linear model.  By assumption \ref{assumptionsparsitynull} and Theorem 5 of \citeA{leung2006},
		\begin{align}\label{equationjronestar}
		\begin{aligned}
			\left\Vert A\mu - AX\betahatewQ \right\Vert^2 &\leq \left\Vert Q\mu - QX\betahatewQ \right\Vert^2 \\&\leq \left\Vert Q \right\Vert^2 \left\Vert \mu - X\betahatewQ \right\Vert^2 = \op(n_a).
		\end{aligned}
		\end{align}
		By the triangle inequality, we have that
		\begin{align}\label{equationjrtwostar}
		\begin{aligned}
			\frac{1}{\sqrt{n_a}}& \left( \left\Vert A\epsilon \right\Vert - \left\Vert A\mu - AX\betahatewQ \right\Vert \right)\\
			\leq& \frac{1}{\sqrt{n_a}} \left\Vert AY - AX\betahatewQ \right\Vert\\
			&\leq \frac{1}{\sqrt{n_a}} \left( \left\Vert A\epsilon \right\Vert + \left\Vert A\mu - AX\betahatewQ \right\Vert \right).
		\end{aligned}
		\end{align}
		By equations (\ref{equationjronestar}) and (\ref{equationjrtwostar}),
		\begin{align*}
			\frac{1}{\sqrt{n_a}} \left\Vert AY - AX\betahatewQ \right\Vert = \frac{1}{\sqrt{n_a}} \left\Vert A\epsilon \right\Vert + \op(1).
		\end{align*}
		Squaring both sides yields
		\begin{align*}
			\frac{1}{n_a} \left\Vert AY - AX\betahatewQ \right\Vert^2 = \frac{1}{n_a} \left\Vert A\epsilon \right\Vert^2 + \op(1).
		\end{align*}
		Similarly, for the denominator,
		\begin{align*}
			\frac{1}{n_b} \left\Vert BY - BX\betahatewQ \right\Vert^2 = \frac{1}{n_b} \left\Vert B\epsilon \right\Vert^2 + \op(1).
		\end{align*}
		Noting that the matrices $A$ and $B$ have orthonormal rows and are mutually orthogonal, independence of $A\epsilon$ and $B\epsilon$ follows from fact that $\epsilon$ is Gaussian.  Dividing finishes the proof.
	\end{proof}
	
	\begin{proof}[Proof of Theorem \ref{theoremalternativehypothesis}]
		Temporarily, we will start by considering the numerator for $\few$.  Let $Q\xi \defined QZ\nu + Q\epsilon \sim \n_v\left( 0_v,QZ\Psi Z^\T Q^\T + \sigma_\epsilon^2 I_v \right)$.  By an identical line of reasoning as in the proof of Theorem \ref{theoremnullhypothesis}, it follows that
		\begin{align*}
			\frac{1}{n_a} \left\Vert AY - AX\betahatewQ \right\Vert^2 = \frac{1}{n_a}\left\Vert A\xi \right\Vert^2 + \op(1/\sqrt{n}).
		\end{align*}
		We would like to remark that the $\op(1/\sqrt{n})$ as opposed to $\op(1)$ in Theorem \ref{theoremnullhypothesis} comes from assumption \ref{assumptionsparsityalternative} and we have applied Lemma \ref{lemmaoraclecorrelated} of the present paper as opposed to Theorem 5 from \citeA{leung2006}.  Similarly, for the denominator,
		\begin{align*}
			\frac{1}{n_b} \left\Vert BY - BX\betahatewQ \right\Vert^2 = \frac{1}{n_b}\left\Vert B\xi \right\Vert^2 + \op(1/\sqrt{n}) = \frac{1}{n_b}\left\Vert B\epsilon \right\Vert^2 + \op(1/\sqrt{n}).
		\end{align*}
		Thus, this yields
		\begin{align*}
			\few &= \frac{\left\Vert A\xi \right\Vert^2/{n_a}}{\left\Vert B\epsilon \right\Vert^2/{n_b}} + \op(1/\sqrt{n})\\
			&\geq \frac{\lambda_{\text{min}}(AZ\Psi Z^\T A^\T) + \sigma_\epsilon^2}{\sigma_\epsilon^2}F_{v,q} + \op(1/\sqrt{n})\\
			&\geq \frac{h/\sqrt{n} + \sigma_\epsilon^2}{\sigma_\epsilon^2}F_{v,q} + \op(1/\sqrt{n}),
		\end{align*}
		We would like to show that
		\begin{align*}
			\e_{H_0} \phidelta + \e_{H_1}\phidelta = 
			\p_{H_0} \left( \few > F_{n_a,n_b,\delta} \right) + \p_{H_1} \left( \few \leq F_{n_a,n_b,\delta} \right) < 1.
		\end{align*}
		It suffices to show that
		\begin{align*}
			\p_{H_1} \left( \few > F_{n_a,n_b,\delta} \right) > \delta.
		\end{align*}
		We will start by providing an upper bound on $F_{n_a,n_b,\delta}$.  From Lemma 1 of \citeA{laurent2000}, it follows that, for a $\chi^2_d$ random variable,
		\begin{align*}
			&\p\left(\chi^2_d > d + 2\sqrt{dx} + 2x\right) \leq \exp(-x),\\
			&\p\left(\chi^2_d \leq d - 2\sqrt{dy} \right) \leq \exp(-y).
		\end{align*}
		Therefore, it follows that, for any $x,y>0$,
		\begin{align*}
			\p\left( F_{n_a,n_b} > \frac{1 + 2\sqrt{x/{n_a}} + 2x/{n_a}}{1 - 2\sqrt{y/{n_b}}} \right) \leq \exp(-x) + \exp(-y).
		\end{align*}
		By choosing $x,y>0$ such that
		\begin{align*}
			\exp(-x) + \exp(-y) \leq \delta,
		\end{align*}
		then
		\begin{align*}
			F_{n_a,n_b,\delta} \leq \frac{1 + 2\sqrt{x/{n_a}} + 2x/{n_a}}{1 - 2\sqrt{y/{n_b}}}.
		\end{align*}
		Let $a,b>0$ be constants that will be chosen later.  Define the event $\Tset$ as
		\begin{align*}
			\Tset \defined \left\{ \chi^2_{n_b} \leq n_b + 2\sqrt{n_bb} + 2b , \chi^2_{n_a} > n_a - 2\sqrt{n_aa} \right\}.
		\end{align*}
		Again, by Lemma 1 of \citeA{laurent2000},
		\begin{align*}
			\p\left( \Tset^\C \right) \leq \exp\left(-a\right) + \exp\left(-b\right).
		\end{align*}
		Now,
		\begin{align*}
			\p_{H_1} \left( \few > F_{n_a,n_b,\delta} \right) &\geq \p_{H_1} \left( \few > F_{n_a,n_b,\delta} , \Tset \right) \\
			&\geq \p_{H_1} \Bigg( \frac{h/\sqrt{n} + \sigma_\epsilon^2}{\sigma_\epsilon^2}\left( \frac{1 - 2\sqrt{a/{n_a}}}{1 + 2\sqrt{b/{n_b}} + 2b/{n_b}} \right)  \\
			&\hphantom{agasdja}+ \op(1/\sqrt{n}) > \frac{1 + 2\sqrt{x/{n_a}} + 2x/{n_a}}{1 - 2\sqrt{y/{n_b}}} , \Tset \Bigg).
		\end{align*}
		From assumption \ref{assumptionsparsityalternative}, we have that $n_a \asymp n_b \asymp n$.  Let
		\begin{align*}
			\liminf_{n\to\infty} \frac{n_a}{n} = c_a, &&& \liminf_{n\to\infty} \frac{n_b}{n} = c_b.
		\end{align*}
		Then, assuming that
		\begin{align*}
			h > 2\sigma_\epsilon^2 \left( \sqrt{bc_b} + \sqrt{ac_a} + \sqrt{yc_b} + \sqrt{xc_a} \right),
		\end{align*}
		it will follow that
		\begin{align*}
			\liminf_{n\to\infty} \p_{H_1} \left( \few > F_{n_a,n_b,\delta} \right) \geq \liminf_{n\to\infty} \p(\Tset) \geq 1 - \exp(-a) - \exp(-b),
		\end{align*}
		which finishes the proof by tuning $a,b>0$ and $h$ sufficiently large appropriately.
	\end{proof}
	
	\begin{proof}[Proof of Lemma \ref{lemmaoraclecorrelated}]
		Let $P_{m}$ denote the projection onto $X_m$ and $\pp{m}$ the projection onto the orthogonal complement.  Let $r_m$ as
		\begin{align*}
			r_m \defined \left\Vert \pp{m} \mu \right\Vert^2.
		\end{align*}
		We will choose an arbitrary sequence of weakly sparse sets $S\in\mathcal{S}_\mu$ satisfying
		\begin{align*}
			r_S = o(n).
		\end{align*}
		Fix $t>0$ arbitrarily and define the set $\A_t$ as follows
		\begin{align*}
			\A_t \defined \left\{ m\in\m_u : r_m \leq tn \right\}.
		\end{align*}
		Looking at the squared norm, we will start with the convexity of the norm and the triangle inequality to obtain that
		\begin{align*}
			\frac{1}{n} \left\Vert \sum_{m\in\m_u} w_m\left( P_{m}Y - \mu \right) \right\Vert^2
			\leq& \frac{1}{n} \sum_{m\in\m_u} w_m \left\Vert P_{m} Y - \mu \right\Vert^2\\
			\leq& \frac{2}{n} \sum_{m\in\m_u} w_m \left( r_m + \left\Vert P_{m} \xi \right\Vert^2 \right)\\
			\leq& \frac{2}{n} \left(\sum_{m\in\A_t} w_m r_m + \sum_{m\in\A_t^\C} w_m r_m \right) \\
			&+ \frac{2}{n} \sum_{m\in\m_u} w_m \left\Vert P_{m} \xi \right\Vert^2.
		\end{align*}
		We will consider each of the three terms separately.  For the first term, it follows from the definition of $\A_t$ that
		\begin{align*}
			\limsup_{n\to\infty} \frac{2}{n} \e\left( \sum_{m\in\A_t} w_m r_m \right) \leq 2t.
		\end{align*}
		For the second term, we have that
		\begin{align*}
			\e \left( \sum_{m\in\A_t^\C} w_m r_m \right)
			= \sum_{m\in\A_t^\C} r_m \e w_m.
		\end{align*}
		Fix $m\in\A_t^\C$ temporarily.  Then, we have that
		\begin{align*}
			w_m &\leq \exp\left( -\frac{1}{\alpha} \left( \left\Vert \pp{m} Y \right\Vert^2 - \left\Vert \pp{S} Y \right\Vert^2 \right) \right)\\
			&\leq \exp\left( -\frac{1}{\alpha} \left( r_m - r_S + 2\xi^\T\pp{m}\mu - 2\xi^\T\pp{S}\mu + \left\Vert P_{S}\xi \right\Vert^2 - \left\Vert P_{m}\xi \right\Vert^2 \right) \right)\\
			&\leq \exp\left( -\frac{1}{\alpha} \left( r_m - r_S + 2\xi^\T\pp{m}\mu - 2\xi^\T\pp{S}\mu - \left\Vert P_{m}\xi \right\Vert^2 \right) \right).
		\end{align*}
		Therefore, by the generalized H\"{o}lder inequality, 
		\begin{align*}
			\e w_m \leq &\exp\left(-\frac{1}{\alpha} \left( r_m - r_S \right) \right) \left(\e \exp\left( -\frac{6}{\alpha} \xi^\T \pp{m} \mu \right)\right)^{1/3} \\
			&\times \left(\e \exp\left( \frac{6}{\alpha} \xi^\T \pp{S} \mu \right)\right)^{1/3} \left(\e \exp\left( \frac{3}{\alpha} \left\Vert P_{m} \xi \right\Vert^2 \right) \right)^{1/3}.
		\end{align*}
		Computing each of the three Laplace transforms separately, since
		\begin{align*}
			\xi^\T \pp{m} \mu \sim \n\left(0 , \left\Vert \Sigma^{1/2} \pp{m} \mu \right\Vert^2 \right),
		\end{align*}
		it follows that
		\begin{align*}
			\left(\e \exp\left( -\frac{6}{\alpha} \xi^\T \pp{m} \mu \right) \right)^{1/3} = \exp \left( \frac{6}{\alpha^2} \left\Vert \Sigma^{1/2} \pp{m} \mu \right\Vert^2 \right) 
			\leq \exp \left( \frac{6\lambda}{\alpha^2} r_m \right).
		\end{align*}
		Similarly,
		\begin{align*}
			\left(\e \exp\left( \frac{6}{\alpha} \xi^\T \pp{S} \mu \right) \right)^{1/3} \leq \exp \left( \frac{6\lambda}{\alpha^2} r_S \right).
		\end{align*}
		For the other term, observe that
		\begin{align*}
			\left\Vert P_{m} \xi \right\Vert^2 \preceq \lambda\chi^2_u,
		\end{align*}
		where $\preceq$ denotes stochastic ordering.  Hence,
		\begin{align*}
			\e \exp\left( \frac{3}{\alpha} \left\Vert P_{m} \xi \right\Vert^2 \right)
			\leq \e \exp\left( \frac{3\lambda}{\alpha} \chi^2_u \right) = \left( 1 - \frac{6\lambda}{\alpha} \right)^{-u/2}.
		\end{align*}
		Since $\alpha > 6\lambda$ and $m\in\A_t^\C$, it will follow that
		\begin{align*}
			\e w_m &\leq 
			\exp\left(-\frac{1}{\alpha} \left(1 - \frac{6\lambda}{\alpha}\right) r_m + \frac{6\lambda}{\alpha^2}r_S - \frac{u}{2}\log\left( 1 - \frac{6\lambda}{\alpha} \right) \right)\\
			&\leq \exp\left(- \left(1 - \frac{6\lambda}{\alpha}\right) \frac{tn}{\alpha} + o(n) - \frac{u}{2}\log\left( 1 - \frac{6\lambda}{\alpha} \right) \right).
		\end{align*}
		By assumption, 
		\begin{align*}
			\limsup_{n\to\infty} \frac{1}{n} \left\Vert \pp{m}\mu \right\Vert^2 \leq \limsup_{n\to\infty} \frac{1}{n} \left\Vert \mu \right\Vert^2 \leq C,
		\end{align*}
		for some constant $C > 0$.  Hence,
		\begin{align*}
			\limsup_{n\to\infty} &\frac{2}{n} \e \left( \sum_{m\in\A_t^\C} w_m r_m \right) \leq 2C \limsup_{n\to\infty} \sum_{m\in\A_t^\C} \e w_m\\
			&\leq 2C\limsup_{n\to\infty} \exp\left(- \left(1 - \frac{6\lambda}{\alpha}\right) \frac{tn}{\alpha} + o(n) - \frac{u}{2}\log\left( 1 - \frac{2\lambda}{\alpha} \right) + \log\left(|\A_t^\C|\right) \right) \\
			&\to 0,
		\end{align*}
		since $\log\left(|\m_u|\right) = o(n)$.  This implies that
		\begin{align*}
			\frac{2}{n} \e \left( \sum_{m\in\A_t^\C} w_m r_m \right) \to 0.
		\end{align*}
		For the last term, fix a value of $K > 0$ large to be determined later.  Define the event $\Tset$ as
		\begin{align*}
			\Tset \defined \bigcap_{m\in\m_u} \left\{ \left\Vert P_{m}\xi \right\Vert^2 \leq K\log\left(|\m_u|\right) + \lambda s \right\}.
		\end{align*}
		We will first provide an upper bound on $\p(\Tset^\C)$.  Indeed,
		\begin{align*}
			\p\left( \Tset^\C \right) \leq \sum_{m\in\m_u} \p\left( \left\Vert P_{m} \xi \right\Vert^2 > K\log\left(|\m_u|\right) + \lambda s \right).
		\end{align*}
		Let $\epsilon \sim \n_n(0_n,I_n)$ and define $Q_m \defined \Sigma^{1/2} P_{m} \Sigma^{1/2}$.  Fixing $m\in\m_u$, we have by the Hanson-Wright inequality (Theorem 1.1 of \citeA{rudelson2013}) that
		\begin{align*}
			\p &\left( \left| \left\Vert P_{m} \xi \right\Vert^2 - \e \left\Vert P_{m} \xi \right\Vert^2 \right| > K\log\left(|\m_u|\right) \right) \\
			&=\p\left( \left| \epsilon^\T Q_m \epsilon - \e \epsilon^\T Q_m \epsilon \right| > K\log\left(|\m_u|\right) \right) \\
			&\leq 2 \exp\left( -c \min\left( \frac{K^2\log^2\left(|\m_u|\right)}{\left\Vert Q_m \right\Vert^2_\text{HS}} , \frac{K\log\left(|\m_u|\right)}{\left\Vert Q_m \right\Vert_\text{op}} \right) \right).
		\end{align*}
		for some universal constant $c>0$.  Note that
		\begin{align*}
			\left\Vert Q_m \right\Vert_\text{op} \leq \left\Vert \Sigma^{1/2} \right\Vert^2_\text{op} \left\Vert P_{m} \right\Vert_\text{op} = \lambda.
		\end{align*}
		Since $Q_m$ is a real symmetric matrix, it admits a Spectral Decomposition with maximal eigenvalue bounded by $\lambda$ and rank at most $s$.  Hence, this gives
		\begin{align*}
			\left\Vert Q_m \right\Vert_\text{HS}^2 \leq \lambda^2 s.
		\end{align*}
		Directly computing the expectation,
		\begin{align*}
			\e \epsilon^\T Q_m \epsilon = \Tr(Q_m) \leq \lambda s.
		\end{align*}
		This gives
		\begin{align*}
			\p& \left( \left| \left\Vert P_{m} \xi \right\Vert^2 - \e \left\Vert P_{m} \xi \right\Vert^2 \right| > K\log\left(|\m_u|\right) \right) \\&\leq 2 \exp \left( -c\min\left( \frac{K^2\log^2\left(|\m_u|\right)}{\lambda^2 s} , \frac{K\log\left(|\m_u|\right)}{\lambda} \right) \right).
		\end{align*}
		This implies that
		\begin{align}
			\begin{aligned}\label{equationtsetprob}
				\sum_{m\in\m_u}&\p\left(  \left\Vert P_{m} \xi \right\Vert^2 > K\log\left(|\m_u|\right) + \lambda s \right) \\&\leq 2\exp\left( -c\min\left( \frac{K^2\log^2\left(|\m_u|\right)}{\lambda^2 s} , \frac{K\log\left(|\m_u|\right)}{\lambda} \right) + \log(|\m_u|) \right).
			\end{aligned}
		\end{align}
		Now, 
		\begin{align*}
			\e \left( \sum_{m\in\m_u} w_m \left\Vert P_m\xi \right\Vert^2 \right)
			= \e \left( \sum_{m\in\m_u} w_m \left\Vert P_m\xi \right\Vert^2 : \Tset \right) + \e \left( \sum_{m\in\m_u} w_m \left\Vert P_m\xi \right\Vert^2 : \Tset^\C \right)
		\end{align*}
		Computing each term separately, we have, by the definition of $\Tset$, that
		\begin{align*}
			\e \left( \sum_{m\in\m_u} w_m \left\Vert P_m\xi \right\Vert^2 : \Tset \right) \leq K\log\left(|\m_u|\right) + \lambda s = o(n).
		\end{align*}
		For the other term, note that
		\begin{align*}
			\e &\left( \sum_{m\in\m_u} w_m \left\Vert P_m\xi \right\Vert^2 : \Tset^\C \right) 
			\leq \e \left( \left\Vert \xi \right\Vert^2 : \Tset^\C \right)\\
			&\leq \left( \left(\e \left\Vert \xi \right\Vert^4 \right) \p\left(\Tset^\C\right) \right)^{1/2}
			\leq \left( \lambda^4 \left(n^2 + 2n \right)\p\left(\Tset^\C\right)\right)^{1/2} \to 0
		\end{align*}
		where the second inequality is due to Cauchy-Schwarz and the limit follows from equation (\ref{equationtsetprob}), assuming that $K>0$ was chosen sufficiently large.  Therefore, we have shown that
		\begin{align*}
			\limsup_{n\to\infty} \frac{1}{n} \e \left\Vert \sum_{m\in\m_u} w_m\left( P_{m}Y - \mu \right) \right\Vert^2 \leq 2t.
		\end{align*}
		Since $t>0$ was arbitrary, this proves the first claim.  For the second claim, note that the proof is identical except for defining the set $\A_t$ as
		\begin{align*}
			\A_t \defined \left\{ m\in\m_u : r_m \leq t\sqrt{n} \right\}.
		\end{align*}
		This finishes the proof.		
	\end{proof}
	
	\subsection{Proofs for Section \ref{sectionci}}
	\begin{proof}[Proof of Theorem \ref{theoremci}]
		Let $\xi \defined D^{-1/2}V^\T AZ\nu + D^{-1/2}V^\T A\epsilon$.  From the proof of Theorem \ref{theoremalternativehypothesis}, it is easy to see that
		\begin{align*}
			&\frac{1}{n_a}\left\Vert D^{-1/2}V^\T AY - D^{-1/2}V^\T AX\betahatewQ \right\Vert^2 = \frac{1}{n_a} \left\Vert \xi \right\Vert^2 + \op(1/\sqrt{n_a}),\\
			&\frac{1}{n_b}\left\Vert BY - BX\betahatewQ \right\Vert^2 = \frac{1}{n_b} \left\Vert B\epsilon \right\Vert^2 + \op(1/\sqrt{n_b}).
		\end{align*}
		Therefore, we immediately have that
		\begin{align*}
			\sqrt{n} \left( \sigmanuhatQ - \dbar\sigma_\nu^2 \right)
			&= \sqrt{n} \left( \frac{1}{n_a} \left\Vert \xi \right\Vert^2 - \frac{1}{n_b} \left\Vert B\epsilon \right\Vert^2 - \dbar\sigma_\nu^2 \right) + \op(1).
		\end{align*}
		We will consider each random variable separately.  Indeed, by the Central Limit Theorem,
		\begin{align*}
			\sqrt{n_b} \left( \frac{1}{n_b} \left\Vert B\epsilon \right\Vert^2 - \sigma_\epsilon^2 \right) \cond \n\left(0,2\sigma_\epsilon^4 \right),
		\end{align*}
		or, equivalently,
		\begin{align*}
			\frac{\sqrt{n_b}}{\sqrt{2\sigma_\epsilon^4}\sqrt{n}} \sqrt{n} \left( \frac{1}{n_b} \left\Vert B\epsilon \right\Vert^2 - \sigma_\epsilon^2 \right) \cond \n(0,1).
		\end{align*}
		For the other term, we will verify Lindeberg's condition.  Define
		\begin{align*}
			&\mu_k \defined \e \xi_k^2 = \sigma_\nu^2 + d_k^{-1} \sigma_\epsilon^2,\\
			&\sigma_k^2 \defined \var(\xi^2_k) = 2\left( \sigma_\nu^2 + d_k^{-1} \sigma_\epsilon^2 \right)^2 = 2\mu_k^2,\\
			&s_{n_a}^2 \defined \sum_{k=1}^{n_a} \sigma_k^2.
		\end{align*}
		Let $\left(\zeta_k\right)_{k=1}^{n_a}$ be a sequence of i.i.d. $\chi^2_1$ random variables.  Then, fixing $\delta >0$, 
		\begin{align*}
			\limsup_{n_a\to \infty} &\frac{1}{s_{n_a}^2} \sum_{k=1}^{n_a} \e \left( \left(\xi_k^2 - \mu_k\right)^2 : \left| \xi_k - \mu_k \right| > \delta s_{n_a} \right)\\
			&\leq \limsup_{n_a\to \infty} \frac{1}{s_{n_a}^2} \sum_{k=1}^{n_a} \mu_k^2 \e \left( \left(\zeta_k^2 - 1\right)^2  : \left| \zeta_k - 1 \right| > \frac{\delta s_{n_a}}{\mu_k} \right)\\
			&\leq \limsup_{n_a\to \infty} \frac{1}{s_{n_a}^2} \sum_{k=1}^{n_a} \mu_k^2 \left(\e \left( \zeta_k^2 - 1 \right)^4\right)^{1/2}  \p\left(\left| \zeta_k - 1 \right| > \frac{\delta s_{n_a}}{\mu_k} \right)^{1/2}\\
			&\leq \limsup_{n_a\to \infty} C \bigvee_{k=1}^{n_a} \p\left(\left| \zeta_k - 1 \right| > \frac{\delta s_{n_a}}{\mu_k} \right)^{1/2}
			\to 0,
		\end{align*}
		where $C$ is the square root of the fourth central moment of a $\chi^2_1$ random variable.  The second inequality is an application of Cauchy-Schwarz and the limit follows from assumption \ref{assumptioneffectivesamplesize}.  Therefore, by the Lindeberg Central Limit Theorem,
		\begin{align*}
			\frac{1}{s_{n_a}} \left( \left\Vert \xi \right\Vert^2 - \left(n_a\sigma_\nu^2 + \Tr(D^{-1})\sigma_\epsilon^2\right) \right) \cond \n\left(0,1\right).
		\end{align*}
		Equivalently,
		\begin{align*}
			\frac{\Tr(D^{-1})}{s_{n_a}\sqrt{n}} \sqrt{n} \left( \frac{1}{\Tr(D^{-1})}\left\Vert \xi \right\Vert^2 - \left(\frac{n_a}{\Tr(D^{-1})}\sigma_\nu^2 + \sigma_\epsilon^2\right) \right) \cond \n\left(0,1\right).
		\end{align*}
		By the independence of $\xi$ and $B\epsilon$ and assumption \ref{assumptionasymptoticsamplesize}, the desired result now follows by differencing.
	\end{proof}

	\begin{proof}[Proof of Proposition \ref{propositionscalingestimation}]
		Indeed, by the Law of Large Numbers,
		\begin{align*}
			\sigmaepsilonhatQ \conp \sigma_\epsilon^2.
		\end{align*}
		Hence, it is clear that $\sigmabhat \conp \sigma_b^2$.  The fact that $\hat{d} \conp \dbar$ is assumption \ref{assumptionasymptoticsamplesize}.  Finally, for $\sigmaahat$, it follows from Theorem \ref{theoremci} that
		\begin{align*}
			\sigmanuhatQ \conp \dbar\sigma_\nu^2.
		\end{align*}
		Now, define $\hat{s}_{n_a}$ as
		\begin{align*}
			\hat{s}^2_{n_a} \defined 2\sum_{k=1}^{n_a} \left( \sigmanuhat/\hat{d} + d_k^{-1}\sigmaepsilonhat \right)^2.
		\end{align*}
		It suffices to show that
		\begin{align*}
			\frac{\hat{s}^2_{n_a}}{s^2_{n_a}} \conp 1,
		\end{align*}
		where $s^2_{n_a}$ is defined in the proof of Theorem \ref{theoremci}.  Then, we have that
		\begin{align*}
			\frac{\hat{s}^2_{n_a}}{s^2_{n_a}}
			&= \frac{\sum_{k=1}^{n_a} \left( \sigmanuhat/\hat{d} + d_k^{-1}\sigmaepsilonhat \right)^2}{\sum_{k=1}^{n_a} \left( \sigma_\nu^2 + d_k^{-1} \sigma_\epsilon^2 \right)^2}\\
			&= \frac{\sum_{k=1}^{n_a} \left( \left(\dbar\sigma_\nu^2 + \op(1)\right)/\left(\dbar + o(1)\right) + d_k^{-1}\left(\sigma_\epsilon^2 + \op(1) \right) \right)^2}{\sum_{k=1}^{n_a} \left( \sigma_\nu^2 + d_k^{-1} \sigma_\epsilon^2 \right)^2}\\
			&= \frac{\sum_{k=1}^{n_a} \left( \left( \sigma_\nu^2 + d_k^{-1}\sigma_\epsilon^2 \right) + d_k^{-1}\op(1) + \op(1) \right)^2}{\sum_{k=1}^{n_a} \left( \sigma_\nu^2 + d_k^{-1} \sigma_\epsilon^2 \right)^2}\\
			&= 1 + \frac{2\sum_{k=1}^{n_a} \left( \sigma_\nu^2 + d_k^{-1}\sigma_\epsilon^2 \right) \left(d_k^{-1}\op(1) + \op(1) \right) }{\sum_{k=1}^{n_a} \left( \sigma_\nu^2 + d_k^{-1} \sigma_\epsilon^2 \right)^2} + \frac{\sum_{k=1}^{n_a} \left(d_k^{-1}\op(1) + \op(1) \right)^2}{\sum_{k=1}^{n_a} \left( \sigma_\nu^2 + d_k^{-1} \sigma_\epsilon^2 \right)^2}.
		\end{align*}
		We will consider each of the two terms separately. Applying Cauchy-Schwarz to the first term, it follows that
		\begin{align*}
			\left|\frac{2\sum_{k=1}^{n_a} \left( \sigma_\nu^2 + d_k^{-1}\sigma_\epsilon^2 \right) \left(d_k^{-1} + \op(1) \right) \op(1)}{\sum_{k=1}^{n_a} \left( \sigma_\nu^2 + d_k^{-1} \sigma_\epsilon^2 \right)^2}\right|
			\leq \frac{2\sqrt{\sum_{k=1}^{n_a}\left(d_k^{-1}\op(1) + \op(1) \right)^2}}{\sqrt{\sum_{k=1}^{n_a} \left( \sigma_\nu^2 + d_k^{-1} \sigma_\epsilon^2 \right)^2}} \conp 0.
		\end{align*}
		For the second term, note that
		\begin{align*}
			\frac{\sum_{k=1}^{n_a} \left( d_k^{-1}\op(1) + \op(1) \right)^2}{\sum_{k=1}^{n_a} \left( \sigma_\nu^2 + d_k^{-1} \sigma_\epsilon^2 \right)^2} \leq \op(1) \frac{1}{\sigma_\epsilon^4} \conp 0,
		\end{align*}
		which finishes the proof.
	\end{proof}
	
	\subsection{Proofs for Section \ref{sectionestimation}}
	
	\begin{proof}[Proof of Proposition \ref{propositionvarianceconsistency}]
		Indeed, the consistency of $\sigmaepsilonhat$ is by the Law of Large Numbers since
		\begin{align*}
			\sigmaepsilonhat = \frac{1}{n_b} \left\Vert B\epsilon \right\Vert^2 + \op(1).
		\end{align*}
		Let $\left(\zeta_k\right)_{k=1}^{n_a}$ be a sequence of i.i.d. $\chi^2_1$ random variables.  For $\sigmanuhat$, we have that
		\begin{align*}
			\frac{1}{\Tr(D^{-1})}&\left\Vert D^{-1/2}V^\T AY - D^{-1/2} V^\T AX\betahatew \right\Vert^2\\
			&=\frac{1}{\Tr(D^{-1})}\left\Vert D^{-1/2}V^\T AZ\nu + D^{-1/2}V^\T A\epsilon \right\Vert^2 + \op(1)\\
			&\equid \frac{1}{\Tr(D^{-1})} \sum_{k=1}^{n_a} \left( \sigma_\nu^2 + d_k^{-1}\sigma_\epsilon^2 \right) \zeta_k + \op(1)\\
			&= \frac{n_a}{\Tr(D^{-1})} \frac{\sigma_\nu^2}{n_a} \sum_{k=1}^{n_a} \zeta_k + \frac{\sigma_\epsilon^2}{\Tr(D^{-1})} \sum_{k=1}^{n_a} d_k^{-1}\zeta_k + \op(1).
		\end{align*}
		Considering each term separately, we have, by the Law of Large Numbers and assumption \ref{assumptionvarianceconsistency},
		\begin{align*}
			 \frac{n_a}{\Tr(D^{-1})} \frac{\sigma_\nu^2}{n_a} \sum_{k=1}^{n_a} \zeta_k \conp \dbar \sigma_\nu^2.
		\end{align*}
		For the second term, by a direct variance calculation,
		\begin{align*}
			\var \left( \frac{\sigma_\epsilon^2}{\Tr(D^{-1})} \sum_{k=1}^{n_a} d_k^{-1}\zeta_k \right) = \frac{\sigma_\epsilon^4}{\Tr\left( D^{-1} \right)^{2}} \sum_{k=1}^{n_a} d_k^{-2} \var(\zeta_k) = \frac{2\sigma_\epsilon^4 \Tr(D^{-2})}{\Tr(D^{-1})^2} \to 0,
		\end{align*}
		where the limit uses assumption \ref{assumptionvarianceconsistency}.  Hence,
		\begin{align*}
			\frac{\sigma_\epsilon^2}{\Tr(D^{-1})} \sum_{k=1}^{n_a} d_k^{-1}\zeta_k \conp \sigma_\epsilon^2.
		\end{align*}
		Combining everything together shows that
		\begin{align*}
			\sigmanuhat \conp \dbar \sigma_\nu^2.
		\end{align*}
		The proof for $\sigmagammahat$ is analogous and is omitted, which finishes the proof.
	\end{proof}
	
	\begin{proof}[Proof of Lemma \ref{lemmaeb}]
		By independence, the joint distribution of $(\nu,\gamma,\epsilon)$ is given by
		\begin{align*}
			\begin{pmatrix}
				\nu \\ \gamma \\ \epsilon
			\end{pmatrix} \sim 
			\n_{v+r+n} \left( 
			\begin{pmatrix}
				0_v \\ 0_r \\ 0_n
			\end{pmatrix} , 
			\begin{pmatrix}
				\sigma_\nu^2 I_v & 0_{v\times r} & 0_{v\times n}\\
				0_{r\times v} & \sigma_\gamma^2 I_r & 0_{r\times n}\\
				0_{n\times v} & 0_{n\times r} & \sigma_\epsilon^2 I_n
			\end{pmatrix}
			\right).
		\end{align*}
		Therefore, the joint distribution of $Y$ and $\eta$ is given by
		\begin{align*}
			\begin{pmatrix}
				Y \\ \eta
			\end{pmatrix} \sim
			\n_{2n} \left( 
			\begin{pmatrix}
				\mu \\ \mu
			\end{pmatrix},
			\begin{pmatrix}
				\sigma_\nu^2 ZZ^\T + \sigma_\gamma^2 WW^\T + \sigma_\epsilon^2 I_n & \sigma_\nu^2 ZZ^\T\\
				\sigma_\nu^2 ZZ^\T & \sigma_\nu^2 ZZ^\T
			\end{pmatrix}
			\right).
		\end{align*}
		By standard results on the conditional mean of $\eta|Y$, it follows that
		\begin{align*}
			\e \left( \eta | Y \right) = \mu + \sigma_\nu^2 ZZ^\T\left( \sigma_\nu^2 ZZ^\T + \sigma_\gamma^2 WW^\T + \sigma_\epsilon^2 I_n\right)^{-1} \left(Y - \mu\right).
		\end{align*}
		which finishes the proof.
	\end{proof}
	
	\begin{proof}[Proof of Theorem \ref{theoremeboracle}]
		From Proposition \ref{propositionvarianceconsistency}, it follows that
		\begin{align*}
			\frac{\sigmanuhat}{\hat{d}} = \sigma_\nu^2 + \delta_\nu^2 &&&\frac{\sigmagammahat}{\hat{\lambda}} = \sigma_\gamma^2 + \delta_\gamma^2 &&\sigmaepsilonhat = \sigma_\epsilon^2 + \delta_\epsilon^2,
		\end{align*}
		where $\delta_\nu^2$, $\delta_\gamma^2$, and $\delta_\epsilon^2$ are all $\op(1)$.  We will write $\delta_\ast^2 \defined \max\{\delta_\nu^2, \delta_\gamma^2, \delta_\epsilon^2\}$.  Define
		\begin{align*}
			&\Sigma \defined \sigma_\nu^2 ZZ^\T + \sigma_\gamma^2 WW^\T + \sigma_\epsilon^2 I_n,\\
			&\Delta \defined \delta_\nu^2 ZZ^\T + \delta_\gamma^2 WW^\T + \delta_\epsilon^2 I_n.
		\end{align*}
		Note that $\Sigma$ is positive definite and $\Delta$ is invertible almost surely.  Then, by the Matrix Inversion Lemma,
		\begin{align*}
			\left( \Sigma + \Delta \right)^{-1} = \Sigma^{-1} - \Sigma^{-1} \left( \Delta^{-1} + \Sigma^{-1} \right)^{-1} \Sigma^{-1}.
		\end{align*}
		Some algebra will yield
		\begin{align*}
			\etahat =& \hat{\mu}_\text{EW} + \frac{\sigmanuhat}{\hat{d}} ZZ^\T \left( \frac{\sigmanuhat}{\hat{d}} ZZ^\T + \frac{\sigmagammahat}{\hat{\lambda}}WW^\T + \sigmaepsilonhat I_n \right)^{-1} \left( Y - \hat{\mu}_\text{EW} \right)\\
			=& \hat{\mu}_\text{EW} + \left( \sigma_\nu^2 + \delta_\nu^2 \right) ZZ^\T \left( \Sigma + \Delta \right)^{-1} \left( Y - \hat{\mu}_\text{EW} \right)\\
			=& \left(\hat{\mu}_\text{EW} - \mu\right) + \left( \sigma_\nu^2 + \delta_\nu^2 \right) ZZ^\T \left( \Sigma + \Delta \right)^{-1} \left( \mu - \hat{\mu}_\text{EW} \right) 
			\\&+ \delta_\nu^2 ZZ^\T \left( \Sigma + \Delta \right)^{-1} \left( Z\nu + W\gamma + \epsilon \right)
			\\&- \sigma_\nu^2 ZZ^\T \Sigma^{-1} \left( \Delta^{-1} + \Sigma^{-1} \right)^{-1} \Sigma^{-1} \left( Z\nu + W\gamma + \epsilon \right)
			\\&+\mu + \sigma_\nu^2 ZZ^\T \Sigma^{-1} \left( Z\nu + W\gamma + \epsilon \right),
		\end{align*}
		where we have added and subtracted $\mu$ and applied the Matrix Inversion Lemma.  From Lemma \ref{lemmaeb}, it follows that
		\begin{align*}
			\etahatoracle = \mu + \sigma_\nu^2 ZZ^\T \Sigma^{-1} \left( Z\nu + W\gamma + \epsilon \right).
		\end{align*}
		Define $\xi \defined \etahat - \etahatoracle$.  Therefore,
		\begin{align*}
			\frac{1}{n} \left( \left\Vert \etahat - \eta \right\Vert^2 - \left\Vert \etahatoracle - \eta \right\Vert^2 \right)
			= \frac{1}{n} \left( \left\Vert \xi \right\Vert^2 + 2\xi^\T\left( \etahatoracle - \eta \right) \right).
		\end{align*}
		We will prove each of the two terms on the right hand side are $\op(1)$.  Before doing so, we will prove a few useful facts that will facilitate the remainder of the proof.
		\begin{enumerate}[label=(\Roman*)]
			\item $\left\Vert ZZ^\T \right\Vert^2 = \mathcal{O}(1)$.
			\item $\left\Vert Z\nu + W\gamma + \epsilon \right\Vert^2 = \Op(n)$ and $\left\Vert Z\nu \right\Vert^2 = \Op(n)$.
			\item $\left\Vert \Delta \right\Vert^2 = \op(1)$ and $\left\Vert \left( \Sigma + \Delta \right)^{-1} \right\Vert^2 = \Op(1)$.
		\end{enumerate}
		(I) is immediate since $\Sigma$ is the sum of three positive semi-definite matrices, one of which being $\sigma_\nu^2 ZZ^\T$, and $\left\Vert \Sigma \right\Vert^2 = \mathcal{O}(1)$ by assumption \ref{assumptionrandomeffectsizes}.
		The first part of (II) follows from the fact that 
		\begin{align*}
			\left\Vert Z\nu + W\gamma + \epsilon \right\Vert^2 \preceq \lambda_{\text{max}}\left( \Sigma \right) \chi^2_n = \Op(n),
		\end{align*}
		where $\preceq$ denotes stochastic ordering and $\lambda_{\text{max}}\left( \Sigma \right) = \mathcal{O}(1)$ by assumption \ref{assumptionrandomeffectsizes}.  The second part of (II) is identical except we apply (I) to bound the maximal eigenvalue.  Finally, for (III), the first part follows from $\delta_\ast^2 \conp 0$ and $\left\Vert \Sigma \right\Vert = \mathcal{O}(1)$.  For the second part, note that
		\begin{align*}
			\left\Vert \left( \Sigma + \Delta \right)^{-1} \right\Vert^2 \leq \left\Vert \Sigma^{-1} \right\Vert^2 \left\Vert \left( I_n + \Sigma^{-1}\Delta \right)^{-1} \right\Vert^2 = \mathcal{O}(1) \left\Vert \left( I_n + \Sigma^{-1}\Delta \right)^{-1} \right\Vert^2.
		\end{align*}
		Then, $\left\Vert \Sigma^{-1}\Delta \right\Vert \conp 0$.  Therefore,
		\begin{align*}
			\left\Vert \left( I_n + \Sigma^{-1}\Delta \right)^{-1} \right\Vert^2
			= \left\Vert \sum_{k=0}^{\infty} \left(- \Sigma^{-1}\Delta \right)^{k} \right\Vert^2
			\leq \sum_{k=0}^{\infty} \left\Vert \left(- \Sigma^{-1}\Delta \right)^{k} \right\Vert^2 < \infty.
		\end{align*}
		This proves all three claims.
		We can now show that $\left\Vert \xi \right\Vert^2 = \op(n)$, which we will show in parts.  To that end, note that
		\begin{align*}
			&\left\Vert \left(I_n - \left( \sigma_\nu^2 + \delta_\nu^2 \right) ZZ^\T \left( \Sigma + \Delta \right)^{-1}\right) \left( \hat{\mu}_\text{EW} - \mu \right) \right\Vert^2\\
			&\leq \left\Vert I_n - \left( \sigma_\nu^2 + \delta_\nu^2 \right) ZZ^\T \left( \Sigma + \Delta \right)^{-1}\right\Vert^2 \left\Vert \hat{\mu}_\text{EW} - \mu \right\Vert^2\\
			&= \op(n).
		\end{align*}
		In the last line, we have used (I) and (III) from above and Lemma \ref{lemmaoraclecorrelated}.  Similarly,
		\begin{align*}
			\Big\Vert &\delta_\nu^2 ZZ^\T \left( \Sigma + \Delta \right)^{-1} \left( Z\nu + W\gamma + \epsilon \right) \Big\Vert^2\\
			&\leq \delta_\nu^4 \left\Vert ZZ^\T \right\Vert^2 \left\Vert \left( \Sigma + \Delta \right)^{-1} \right\Vert^2 \left\Vert Z\nu + W\gamma + \epsilon \right\Vert^2 \\&= \op(1).
		\end{align*}
		The last line follows from the fact that $\delta_\nu^2 \conp 0$ and (I), (II), and (III) from above.  For the last term, we will apply the Matrix Inversion Lemma again to obtain
		\begin{align*}
			\left( \Delta^{-1} + \Sigma^{-1} \right)^{-1} = \Delta - \Delta \left( \Delta + \Sigma \right)^{-1} \Delta.
		\end{align*}
		Hence,
		\begin{align*}
			&\left\Vert \sigma_\nu^2 ZZ^\T \Sigma^{-1} \left( \Delta^{-1} + \Sigma^{-1} \right)^{-1} \Sigma^{-1} \left( Z\nu + W\gamma + \epsilon \right) \right\Vert^2\\
			&\leq \sigma_\nu^4 \left\Vert ZZ^\T \right\Vert^2 \left\Vert \Sigma^{-1} \right\Vert^4 \left\Vert \Delta \right\Vert^2 \left\Vert I_n - \left( \Delta + \Sigma \right)^{-1} \Delta \right\Vert^2 \left\Vert Z\nu + W\gamma + \epsilon \right\Vert^2\\
			&= \op(n),
		\end{align*}
		where (I), (II), and (III) have all been used from above.  Combining these three results with the triangle inequality, this proves that
		\begin{align*}
			\left\Vert \xi \right\Vert^2 = \op(n).
		\end{align*}
		For the other quantity, we have that
		\begin{align*}
			\left\Vert \etahatoracle - \eta \right\Vert^2 &= \left\Vert \sigma_\nu^2 ZZ^\T \Sigma^{-1} \left( Z\nu + W\gamma + \epsilon \right) - Z\nu \right\Vert^2\\
			&\leq 2 \left\Vert \sigma_\nu^2 ZZ^\T \Sigma^{-1} \left( Z\nu + W\gamma + \epsilon \right) \right\Vert^2 + 2\left\Vert Z\nu \right\Vert^2\\
			&\leq 2 \left\Vert \sigma_\nu^2 ZZ^\T \Sigma^{-1} \left( Z\nu + W\gamma + \epsilon \right) \right\Vert^2 + 2\left\Vert Z\nu \right\Vert^2\\
			&\leq 2 \sigma_\nu^4 \left\Vert ZZ^\T \right\Vert^2 \left\Vert \Sigma^{-1} \right\Vert^2 \left\Vert  Z\nu + W\gamma + \epsilon \right\Vert^2 + 2\left\Vert Z\nu \right\Vert^2.
		\end{align*}
		Applying all three facts from above demonstrates that
		\begin{align*}
			\left\Vert \etahatoracle - \eta \right\Vert^2 = \Op(n).
		\end{align*}
		Using the Cauchy-Schwarz inequality yields
		\begin{align*}
			2\left|\xi^\T\left( \etahatoracle - \eta \right)\right| \leq 2\left\Vert \xi \right\Vert \left\Vert \etahatoracle - \eta \right\Vert = \op(n),
		\end{align*}
		which finishes the proof.
	\end{proof}

\bibliographystyle{newapa}
\bibliographystyle{acm}	
\bibliography{mixedmodelref}
	
\end{document}